\documentclass[12pt,leqno]{article}
\usepackage{amssymb,amsmath,amsthm}
\usepackage[all]{xy}
\usepackage[alphabetic,lite]{amsrefs}
\numberwithin{equation}{section}
\usepackage[top=2cm,bottom=2cm,left=2cm,right=4cm,marginparsep=0.3cm,marginparwidth=3cm,includefoot]{geometry}
\usepackage{mathtools}
\usepackage{persi}
\usepackage{enumerate}
\usepackage[colorlinks=true, pdfstartview=FitV, linkcolor=blue, citecolor=blue, urlcolor=blue]{hyperref}
\usepackage[normalem]{ulem}  % for strikeout \sout
\usepackage{mathrsfs}
\usepackage{accents}
\usepackage{xspace}
\usepackage{marginnote}

\begin{document}

\title{Persistent homology and microlocal sheaf theory}
\author{Masaki Kashiwara and Pierre Schapira%
}

\maketitle

\begin{abstract}
We interpret some results of persistent homology and barcodes (in any dimension) with the language of microlocal sheaf theory. For that purpose we study the derived category of  sheaves on 
 a real  finite-dimensional  vector space $\BBV$. By using the operation of convolution, we introduce a pseudo-distance on this category and prove in particular  a stability result for direct images. Then we assume that 
 $\BBV$ is endowed with  a closed convex proper cone $\gamma$ with non empty interior and study $\gamma$-sheaves, that is, constructible sheaves  with microsupport contained in the antipodal to the polar cone (equivalently, constructible sheaves for the $\gamma$-topology). We prove that such sheaves may be approximated 
(for the pseudo-distance) by   ``piecewise linear'' $\gamma$-sheaves. 
Finally we show  that  these last sheaves are  constant 
on stratifications by $\gamma$-locally closed sets, an analogue of barcodes in higher dimension.  \end{abstract}
{\renewcommand{\thefootnote}{\mbox{}}
\footnote{The research of M.K
was supported by Grant-in-Aid for Scientific Research (B)
15H03608, Japan Society for the Promotion of Science.}
\footnote{The  research of P.S was supported by the  ANR-15-CE40-0007 ``MICROLOCAL''.}
\addtocounter{footnote}{-2}
}

\tableofcontents

\section*{Introduction}

Persistent homology and barcodes are recent concrete applications of algebraic topology.
The aim of this paper is to show that many results of this theory are easily interpreted in the language of microlocal sheaf theory and that, in this formulation, one may extend the theory to higher dimension.

Although the theory is quite new, there is already a vast  literature on
persistent homology. See in particular the survey papers~\cites{EH08, Gri08, Ou15}.

We understand persistent homology as follows. One has a  
finite subset $S$ of an Euclidian space $X$
and one wants to understand its structure. For that purpose, one replaces each point $x\in S$ with a closed ball of center $x$ and radius $t$ and makes $t$ going to infinity. The union of these balls  gives 
a closed set $Z\subset X\times\R$, 
and one wants to understand 
how the homology of the union of the balls varies with $t$, which
is equivalent to understanding
the direct image of the constant sheaf  $\cor_Z$ associated with $Z$
by the projection $X\times\R\to\R$. 
{}From this point of view, we are not far from Morse theory for sheaves (see~\cites{GM88,KS90}).  Moreover, the sheaf one obtains has particular properties. It is constructible and is associated with a topology whose open sets are the intervals $]-\infty,a[$ with  $a\in[-\infty,+\infty]$. As we shall see, the category of such sheaves is equivalent to a category (that we shall define) of barcodes.  Note that the idea of using sheaf theory in this domain is not new and already appeared in the thesis of Justin Curry~\cite{Cu13}.

As described above,  persistent homology takes  its values on $\R$ and barcodes are defined on the ordered space $(\R,\leq)$. However,  the necessity of treating more than one parameter $t$ naturally  appears 
(see for {\em e.g.}~\cites{Le15, LW15}).

A higher dimensional generalization of the ordered space $(\R,\leq)$ is the data of a finite-dimensional  real vector space 
$\BBV$ and a closed convex proper cone $\gamma\subset\BBV$  with non-empty interior. 
We call here a $\gamma$-sheaf an object of the derived category of sheaves on $\BBV$ whose microsupport 
is contained in $\gammac$, the antipodal to the polar cone to $\gamma$. 
As we shall see, this category is equivalent to the derived category of sheaves on $\BBV_\gamma$, 
the space $\BBV$ endowed with the so-called $\gamma$-topology. 

The main goal of this paper is to describe constructible $\gamma$-sheaves on $\BBV$.
Constructible sheaves on  real analytic manifolds are now well-understood (the story began on complex manifolds with~\cite{Ka75}) but, as we shall see, 
$\gamma$-sheaves have a very specific behavior and are not so easy to treat. We shall mainly restrict ourselves to piecewise linear sheaves ($\PL$-sheaves for short), those which  are locally constant on a stratification associated with a locally finite family of hyperplanes.

\medskip
The main results of this paper may be described as follows.
\banum
\item
In  \S\,\ref{section:persi} we first recall and complete some results of~\cite{KS90}*{\S\,3.5} on the $\gamma$-topology,
 showing that the category of $\gamma$-sheaves is equivalent to the derived category of sheaves on $\BBV_\gamma$. 
Then  we recall, in our language, some basic results of persistent homology (that is, essentially Morse theory for sheaves) and construct the category of barcodes on $\R_\gamma$,  for $\gamma=\R_{\leq0}$.  By proving a variant of a result of Crawley-Boevey~\cite{CB14}*{Th.~1.1}, using  a  result of Guillermou~\cite{Gu16}*{Cor.~7.3} based on Gabriel's theorem on representations of quivers, we prove that our category of barcodes is  equivalent to that of constructible $\gamma$-sheaves.
 \item
In \S\,\ref{section:dist}, we define a  kind of pseudo-distance  $\dist$ on sheaves on $\BBV$, after having endowed $\BBV$ with a norm $\vvert\cdot\vvert$. The main tool for this construction is the convolution of sheaves and the main difference between our distance and the other classical ones, is that our distance  is constructed in the derived setting. For example, two sheaves concentrated in different degrees may have a finite distance.
We prove a stability result for direct images, namely, given two continuous maps $f_1,f_2\cl X\to\BBV$ and a sheaf $F$ on $X$, then $\dist(\roim{f_1}F,\roim{f_2}F)\leq\vvert f_1-f_2\vvert$. 

Next,  we introduce  the notion of \PL-sheaves
and prove that any constructible sheaf  may be  approximated  by  a $\PL$-sheaf as well as  a similar result for 
$\PL$-$\gamma$-sheaves (now, $\gamma$ is polyhedral).

Finally, we propose a notion of barcodes in higher dimension and construct the category of such barecodes. Unfortunately, although   this category is naturally embedded  into that of $\PL$-sheaves, we show on examples that this embedding is not an equivalence.

\item
In  \S\,\ref{section:PL}, we first  study with some details the $\gamma$-topology on $\BBV$ and constructible $\gamma$-sheaves. For example, we show that if $F$ is such a sheaf, then for any $x\in\BBV$, $F$ is locally constant on $x+\gamma^a$ in a neighborhood of $x$. 
 
The main result of this section is that  given a $\PL$-$\gamma$-sheaf, there exists a stratification of 
$\BBV$ by $\gamma$-locally closed polytopes on which the sheaf is constant. 
\eanum

\vspace{1.ex}\noindent
{\bf Acknowledgments}\\
The second named author warmly thanks  Gregory Ginot for having organized a seminar  on  persistent homology, at the origin of this paper,  and Beno\^it Jubin for fruitful discussions on this subject. In this seminar,  Nicolas Berkouk and  Steve Oudot pointed out the problem of approximating 
constructible sheaves with objects which would be similar to higher dimensional barcodes, what we do, in some sense,  here. 
Moreover, the links between $\gamma$-sheaves and persistent modules (Proposition~\ref{pro:be}) were clarified during discussions with Nicolas Berkouk.

\section{Persistent homology}\label{section:persi}

\subsection{Sheaves}
The aim of this subsection is simply to fix a few notations.
\begin{itemize}
\item
Throughout the paper, $\cor$ denotes a field.
We denote by $\md[\cor]$ the abelian category of $\cor$-vector spaces. 
\item
For an abelian category $\shc$, we denote by $\Derb(\shc)$ its bounded derived category. However, 
we write $\Derb(\cor)$ instead of $\Derb(\md[\cor])$.
\item
If $\pi\cl E\to M$ is a vector bundle over $M$, we identify $M$ with the zero-section of $E$ and we set
$\sdot{E\;}\eqdot E\setminus M$. We denote by $\sdot\pi\cl \sdot{E\;}\to M$ the restriction of $\pi$ to $\sdot{E\,}$. 
\item
For a vector bundle $E\to M$, we denote by $a\cl E\to E$ the antipodal map, $a(x,y)=(x,-y)$.
For a subset $Z\subset E$, we simply denote by $Z^a$ its image by the antipodal map.
In particular, 
for a cone $\gamma$ in $E$, we denote by $\gamma^a=-\gamma$ the opposite cone. For such a cone, we denote by $\gamma^\circ$ the polar cone (or dual cone) in the dual vector bundle $E^*$:
\eq\label{eq:polar}
&&\gamma^\circ=\{(x;\xi)\in E^*;\langle\xi,v\rangle\geq0\mbox{ for all }v\in\gamma_x\}.
\eneq
\item
Let $M$ be a real manifold $M$ of dimension $\dim M$. We shall use freely the classical notions of microlocal sheaf theory, referring to~\cite{KS90}.
We denote by $\md[\cor_M]$ the abelian category of sheaves of $\cor$-modules on $M$ and by $
\Derb(\cor_M)$ its bounded derived category.
For short, an object of $\Derb(\cor_M)$ is called a ``sheaf'' on $M$.
\item
For a locally  closed subset $Z\subset M$, one denotes by $\cor_Z$ the constant sheaf with stalk $\cor$ on $Z$ extended by $0$ on $M \setminus Z$. One defines similarly the sheaf  $L_Z$ for $L\in\Derb(\cor)$.
\item 
We denote by $\ori_M$ the orientation sheaf on $M$ and by $\omega_M$ the dualizing complex on $M$. Recall that 
$\omega_M\simeq\ori_M\,[\dim M]$. One shall use the duality functors 
\eq\label{eq:dualfct}
&& \RD'_M(\scbul)=\rhom(\scbul,\cor_M),\quad  \RD_M(\scbul)=\rhom(\scbul,\omega_M). 
\eneq
\item
For $F\in\Derb(\cor_M)$ we denote by $\musupp(F)$\footnote{$\musupp(F)$ was denoted by $\SSi(F)$ in  \cite{KS90}.} its microsupport, a closed conic co-isotropic subset of $T^*M$.
\item
For $F\in\Derb(\cor_M)$, one denotes by $\Sing(F)$ the singular locus of $F$, that is, 
the complement of the  largest open  subset on which $F$ is locally constant. 
\end{itemize}

\subsubsection*{Constructible sheaves}
We refer the reader to \cite{KS90} for terminologies not explained here. 
\begin{definition}
Let $M$ be a real analytic manifold and let $F \in \md[\cor_M]$.
One says that $F$ is weakly $\R$-constructible if there exists a subanalytic stratification $M=\bigsqcup_\alpha M_\alpha$ such that for each  stratum $M_\alpha$, the restriction $F\vert_{M_\alpha}$ is locally constant.
If moreover, the stalk $F_x$ is of finite rank for all $x\in M$, then one says that $F$ is $\R$-constructible.
\end{definition}
\begin{notation}
(i) One denotes by $\mdrc[\cor_M]$ the abelian category of $\R$-constructible sheaves and by $\mdrcc[\cor_M]$ the full subcategory of $\mdrc[\cor_M]$ consisting of sheaves with compact support. Both are thick abelian subcategories of $\md[\cor_M]$.  

\spa
(ii) One denotes by $\Derb_{\Rc}(\cor_{M})$ the full triangulated subcategory of $\Derb(\cor_{M})$ consisting of sheaves with $\R$-constructible cohomology  and by  $\Derb_{\rcc}(\cor_{M})$ the full triangulated subcategory of $\Derb_\Rc(\cor_{M})$ consisting of sheaves with compact support.
\end{notation}
A theorem of~\cite{Ka84} (see also~\cite{KS90}*{Th.~8.4.5}) asserts that the natural functor $\Derb(\mdrc[\cor_M])\to\Derb_\Rc(\cor_{M})$ is an equivalence of categories. 

When $F\in \Derb_\Rc(\cor_{M})$, $\musupp(F)$ and $\Sing(F)$ are subanalytic. The first result is proved in loc.\ cit.\ and the second one follows from $\Sing(F)=\sdot{\pi}(\musupp(F)\cap\dTM)$.
(Recall that $\sdot\pi$ is the projection $\dTM\seteq
T^*M\setminus M\to M$.)

\subsection{$\gamma$-topology}\label{subsection:gtop}
The so-called $\gamma$-topology has been studied with some details in~\cite{KS90}*{\S\,3.4}.

Let $\BBV$ be a  finite-dimensional  real vector space.
We denote by $s$ the addition map.
\eqn
&&s\cl \BBV\times\BBV\to\BBV,\quad (x,y)\mapsto x+y,
\eneqn
and by $a\cl \BBV\to\BBV$ the antipodal map $x\mapsto -x$.

Hence,  for two subsets $A,B$ of $\BBV$, one has $A+B=s(A\times B)$. A subset $A$ of $V$ is called a cone if $0\in A$ and $\R_{>0}A\subset A$. A convex cone  $A$ is proper
if $A\cap A^a=\{0\}$.

Throughout  the paper, we consider a cone $\gamma\subset\BBV$ 
 and we assume:
 \eq\label{hyp1}
&&\parbox{75ex}{\em{
$\gamma$ is  closed proper convex with non-empty interior. }
}\eneq
Sometimes we shall make the extra assumption that $\gamma$ is {\em polyhedral}, meaning that it is a finite intersection of closed  half-spaces.

We say that a subset $A$ of $\BBV$ is {\em $\gamma$-invariant} if $A+\gamma=A$.
Note that a subset $A$ is $\gamma$-invariant if and only if
$\BBV\setminus A$ is $\gamma^a$-invariant.

The family of $\gamma$-invariant open subsets of $\BBV$ defines a topology, which is called  the \define{$\gamma$-topology}\footnote{See Subsection~\ref{subsection:alex} for related constructions.}  on $\BBV$.
One denotes by $\BBV_\gamma$ the space $\BBV$ endowed with the $\gamma$-topology and one denotes by
\begin{equation}
\phig \cl\BBV \to \BBV_\gamma
\end{equation}
the continuous map associated with the identity.
Note that the closed sets for this topology are the 
$\gamma^a$-invariant closed subsets of $\BBV$. 
 
\begin{definition}\label{def:loclo}
Let $A$ be a subset of $\BBV$. 
\banum
\item
One says that $A$ is {\em$\gamma$-open} \lp resp.\ {\em $\gamma$-closed}\rp, if $A$ is open \lp resp.\  closed\rp\,  for the $\gamma$-topology. 
\item
One says that $A$ is {\em $\gamma$-locally closed} if $A$ is the intersection of  a $\gamma$-open subset and a $\gamma$-closed  subset.

\item
One says that $A$ is {\em $\gamma$-flat} if $A=(A+\gamma)\cap(A+\gamma^a)$.
\item 
One says that a closed set $A$ is {\em $\gamma$-proper} if the map $s$ is proper on $A\times\gamma^a$.
\eanum
\end{definition}

\begin{remark}
(i) A closed subset $A$ is $\gamma$-proper if and only if
$A\cap(x+\gamma)$ is compact for any $x\in\BBV$. 

\spa
(ii) Let $A$ be a subset of $\BBV$ and assume that  $\ol{A}$ is $\gamma$-proper.
 Then $\ol{A+\gamma^a}=\ol{A}+\gamma^a$.

\spa
(iii) If $A$ is closed and if there exist a closed convex proper cone 
$\gamma_1$ with $\gamma\subset\Int(\gamma_1)\cup\{0\}$ and $x\in\BBV$ such that 
$A\cap(x+\gamma_1)=\varnothing$, then $A$ is $\gamma$-proper.

\spa
(iv) One has $\Int(\gamma)=\Int(\gamma)+\gamma$ and
$\ol{\Int(\gamma)}=\gamma$.
\end{remark}

We shall use the notations:
\eq\label{not:1}
&&\left\{\begin{array}{l}
\Derb_{\gammac}(\cor_{\BBV})\eqdot\{F\in\Derb(\cor_{\BBV}) ;\musupp(F)\subset \BBV\times\gammac\},\\[1ex]
\Derb_{\rcg}(\cor_{\BBV})\eqdot\Derb_{\Rc}(\cor_{\BBV})\cap\Derb_{\gammac}(\cor_{\BBV}),\\[1ex]
\mdg[\cor_{\BBV}]\eqdot\md[\cor_{\BBV}]\cap \Derb_{\gammac}(\cor_\BBV),\\[1ex]
\mdrcg[\cor_{\BBV}]\eqdot\mdrc[\cor_\BBV]\cap\mdg[\cor_{\BBV}].
\end{array}\right.
\eneq
We call an object of $\Derb_{\gammac}(\cor_{\BBV})$ \define{a $\gamma$-sheaf}. 

It follows from~\cite{KS90}*{Prop.~5.4.14} that for $F,G\in\Derb_{\gammac}(\cor_{\BBV})$ and $H\in\Derb_{\gamma^\circ}(\cor_{\BBV})$, the sheaves $F\tens G$ and $\rhom(H,F)$ belong to $\Derb_{\gammac}(\cor_{\BBV})$.

The next result is implicitly proved in~\cite{KS90}, without assuming that $\Int(\gamma)$ is non empty.

\begin{theorem}\label{th:eqvderbg}
Let $\gamma$ be a closed convex proper cone in $\BBV$.
The  functor $\roim{\phig}\cl \Derb_{\gammac}(\cor_\BBV)\to\Derb(\cor_{\BBV_\gamma})$ is an equivalence of triangulated categories with quasi-inverse $\opb{\phig}$.
\end{theorem}
\begin{proof}
(i) The proof of~\cite{KS90}*{Prop.~5.2.3} shows that $\opb{\phig}\cl \Derb(\cor_\Vg)\to \Derb_{\gammac}(\cor_\BBV)$ is well defined. The same statement asserts that for $F\in\Derb_{\gammac}(\cor_\BBV)$, one has 
$\opb{\phig}\roim{\phig}F\isoto F$.

\spa
(ii) By~\cite{KS90}*{Prop.~3.5.3~(iii)}, for $F\in\Derb(\cor_\Vg)$, one has 
$F\isoto \roim{\phig}\opb{\phig}F$.
\end{proof}

\begin{corollary}\label{cor:eqvmdg}
The  functor $\oim{\phig}\cl \mdg[\cor_\BBV]\to\md[\cor_{\Vg}]$ is an equivalence of abelian categories with quasi-inverse $\opb{\phig}$. 
\end{corollary}
\begin{proof}
(i)  By Theorem~\ref{th:eqvderbg}, the functor $\opb{\phig}\cl \md[\cor_\Vg]\to \md[\cor_\BBV]$ takes its values in $\mdg[\cor_\BBV]$. Moreover, the same statement asserts that for $F\in\md[\cor_\Vg]$, $F\isoto \roim{\phig}\opb{\phig}F$. Therefore,  $F\isoto \oim{\phig}\opb{\phig}F$. 

\spa
(ii)  By taking the $0$-th cohomology of the isomorphism $\opb{\phig}\roim{\phig}F\isoto F$ and using the fact that $\opb{\phig}$ commutes with $H^0$, we get the isomorphism $\opb{\phig}\oim{\phig}F\isoto F$.
\end{proof}

\begin{corollary}\label{cor:eqvderbg}
The equivalence of categories in {\rm Theorem~\ref{th:eqvderbg}} preserves the natural $t$-structures of both categories. In particular, 
for $F\in \Derb(\cor_\BBV)$, the condition
$F\in \Derb_{\gammac}(\cor_\BBV)$ is equivalent to the condition:
$\musupp(H^j(F))\subset \gamma^{\circ a}$  for any $j\in\Z$.
\end{corollary}
\begin{proof}
This follows from Corollary~\ref{cor:eqvmdg}.
\end{proof}

\begin{corollary}\label{cor:ssA}
Let $A$ be a $\gamma$-locally closed subset of \/ $\BBV$.
Then $\musupp(\cor_A)\subset \BBV\times\gammac$. 
\end{corollary}
\begin{proof}
The subset $A$ is locally closed in $V_\gamma$.
Let us denote by $\cor_{A,\gamma}\in\Mod(\cor_{V_\gamma})$ the constant sheaf 
supported on $A$.
Then
we have
$\cor_A\simeq \opb{\phig}(\cor_{A,\gamma})$. 
\end{proof}

\begin{remark}
Thanks to Theorem~\ref{th:eqvderbg}, the reader may ignore microlocal sheaf theory, at least in a first reading. Indeed, if this theory plays a central role in the proofs of the statements, it does not  appear in the statements, after replacing $\Derb_{\gammac}(\cor_{\BBV})$ with 
$\Derb(\cor_{\BBV_\gamma})$.
\end{remark}

\subsection{Persistent homology}\label{subsection:persi}

Let $\BBV$ be a real finite-dimensional vector space
and let $\gamma$ be a  cone satisfying hypothesis~\eqref{hyp1}. We also  assume that $\gamma$ is subanalytic\footnote{In practice the cone $\gamma$ will be polyhedral.}.

Let $M$ be a real analytic manifold that 
 and let $f\cl M\to\BBV$ be a continuous subanalytic map. 
We denote by $\Gamma^+_f\subset M\times\BBV$ the $\gamma$-epigraph of $f$.
\eqn
\Gamma^+_f&=&\{(x,y)\in M\times\BBV; f(x)-y\in\gamma\}\\
&=&\Gamma_f+\gamma^a.
\eneqn
We denote by $p\cl M\times\BBV\to \BBV$ the projection. 

\begin{lemma}\label{le:persihom2}
One has $\musupp(\cor_{\Gamma^+_f})\subset T^*M\times(\BBV\times\gammac)$. 
\end{lemma}
\begin{proof}
The set $\Gamma^+_f$ being $\gamma$-closed, the result follows from Corollary~\ref{cor:ssA}.
\end{proof}

\begin{theorem}\label{th:persihom}
Let  $M$ be a real analytic manifold and let $f\cl M\to\BBV$ be a continuous subanalytic map. 
Assume that
\eq\label{hyp:persi1}
\parbox{70ex}{
for each $K\subset\BBV$ compact, the set $\{x\in M;f(x)\in K+\gamma\}$ is compact.
}\eneq
 Then $\roim{p}\cor_{\Gamma^+_f}$ belongs to $\Derb_{\rcg}(\cor_\BBV)$.
\end{theorem}
\begin{proof}
Let $K\subset\BBV$ be a compact subset. Then 
\eqn
\{(x,y)\in \Gamma^+_f;y\in K\}&=&\{(x,y)\in M\times\BBV;f(x)\in y+\gamma, y\in K\} \\
&\subset& \{x\in M;f(x)\in K+\gamma\}\times K.
\eneqn
Hence,  the map $p$ is proper on $\Gamma^+_f$.

Applying~\cite{KS90}*{Prop.~5.4.4}, we get that  $\roim{p}\cor_{\Gamma^+_f}$ belongs to 
$\Derb_{\gammac}(\cor_\BBV)$. Moreover, this object is $\R$-constructible by loc.\ cit.\ Prop.~8.4.8.
\end{proof}

\begin{example}\label{exa:persi1}
Let $M$ and $f$ be as above with  $\BBV=\R$ and $\gamma=\{t\leq0\}$.  
In this case,  $\Gamma^+_f=\{(x,t)\in M\times\R;f(x)\leq t\}$ is the epigraph of $f$.
Hypothesis~\eqref{hyp:persi1} is translated as:
\eq\label{hyp:persi2}
&&\parbox{70ex}{
for each $t\in\R$, the set $\{x\in M;f(x)\leq t\}$ is compact. 
}\eneq
 Set $K=\{x\in M;f(x)\leq0\}$. By the hypothesis, $K$ is compact. Let $a=\inf_{x\in K}f(x)$. Then 
$f(x)\geq a$ for all $x\in M$.

A more explicit example may be obtained as follows. Assume  that  $M$ is endowed with a subanalytic  distance  and let $S$ be a finite subset of $M$. Then one can choose $f(x)=d(x,S)$ in which case, 
\eqn
&&\Gamma^+_f=\{(x,t)\in M\times\R;d(x,S)\leq t\}=\bigcup_{s\in S}B(s;t)
\eneqn
where $B(s;t)$ is the closed ball of center $s$ and radius $t$. One can also endow each $s\in S$ with some weight 
$\rho(s)\in\R_{\geq0}$ and replace $B(s;t)$ with $B(s;{\rho(s)t})$.
\end{example}

\begin{example}\label{exa:persi2}
Let $\BBV=\R^n$, $\gamma=(\R_{\leq0})^n$ and let $f=(f_1,\cdots,f_n)$, each $f_i$ being continuous and subanalytic and satisfying
 hypothesis~\eqref{hyp:persi2}. Then $f$ satisfies the hypothesis of Theorem~\ref{th:persihom}. Moreover, in this case, 
 $\supp(\roim{p}\cor_{\Gamma^+_f})\subset y+\gamma^a$ for some $y\in\BBV$.
\end{example}

\begin{remark}
We have assumed that $M$ is real analytic and $f\cl M\to\BBV$ is continuous and subanalytic in order that 
$\roim{p}\cor_{\Gamma^+_f}$ be constructible. If $\BBV=\R$,   these hypotheses may be weakened by simply assuming that the critical values of $f$ are discrete. 

Also note that, although different from the class of $C^\infty$-maps, the class of continuous subanalytic maps is very large.
\end{remark}

\subsection{ Comparison with the Alexandrov topology}\label{subsection:alex}
Most of the authors define (higher dimensional) persistent homology modules as functors from the order set 
$(\R^n,\leq)$ (where $\leq$ is the product order corresponding to the cone $\gamma=(\R_{\leq0})^n$) to the category of  vector spaces over the field $\cor$. Some authors also consider the Alexandrov topology on $\R^n$ associated with $\gamma$. Let us clarify the link with these different approaches.

Set $\BBV$, $\gamma$ and $\Vg$ be as above. Denote by $\Wg$ the set $\BBV$ endowed with the Alexandrov topology associated with the cone $\gamma$. Hence,  $A\subset\BBV$ is  open in $\Wg$ if and only if $A=A+\gamma$. On the other hand denote by $\Xg$ the set $\BBV$ endowed with the order $x\leq y$ if and only if $x+\gamma\subset y+\gamma$. We look at $\Xg$ as a category, hence as a presite, hence as a site for the trivial Grothendieck topology. Therefore,
a sheaf on $\Xg$ is nothing but a presheaf on $\Xg$, that is, a functor $\Xg^\rop\to\md[\cor]$  and, of course, $\Xg^\rop$ is equivalent to $(\BBV,\leq^\rop)$. 

We have morphisms of sites
\eqn
&&\alpha\cl \Vg\to\Wg,\quad \Op_\Wg\ni A+\gamma\mapsto A+\Int(\gamma)\in\Op_\Vg,\\
&&\beta\cl \Wg\to\Xg,\quad \Xg\ni x\mapsto x+\gamma,\\
&&\rho\cl \Vg\to \Xg\quad \Xg\ni x\mapsto x+\Int(\gamma).
\eneqn
Note that $\rho\simeq \beta\circ\alpha$.  The morphisms $\alpha$, $\beta$ and $\rho$ induce functors
\eqn
\xymatrix{
\md[\Vg]\ar@<0.4ex>[r]^-{\oim{\alpha}}&\md[\Wg]\ar@<0.4ex>[l]^-{\opb{\alpha}},
}
\xymatrix{
\md[\Wg]\ar@<0.4ex>[r]^-{\oim{\beta}}&\md[\Xg]\ar@<0.4ex>[l]^-{\opb{\beta}},
}
\xymatrix{
\md[\Vg]\ar@<0.4ex>[r]^-{\oim{\rho}}&\md[\Xg]\ar@<0.4ex>[l]^-{\opb{\rho}}.
}
\eneqn
As already mentioned, the category of persistent modules  is the category 
of  presheaves (equivalently, of sheaves) on $\Xg$.
Consider the diagram of categories
\eqn
\xymatrix{
\mbox{$\md[\Wg]\rule[-1.0ex]{0ex}{2ex}$}\ar[rr]^-{\oim{\beta}}_-\sim&&\md[\Xg]\\
&\mbox{$\hs{1.5ex}\md[\Vg]\hs{.0ex}\rule{0ex}{2.0ex}$}%
\ar@<-0.ex>@{_{(}->}[lu]_-{\oim{\alpha}}\ar@{^{(}->}[ru]_-{\oim{\rho}}&
}
\eneqn

\begin{proposition}\label{pro:be}\footnote{As already mentioned, these  results were clarified during discussions of the second named author with Nicolas Berkouk.}
\banum
\item
The functor  $\oim{\rho}$ is fully faithful
and  is  not an equivalence. 
\item
The functor   $\oim{\beta}$ is an equivalence.\footnote{This statement is due to~\cite{Cu13}*{Th.4.2.10}. }
\item
The functor $\oim{\alpha}$ is fully faithful. 
\eanum
\end{proposition}
\begin{proof}
In the sequel, for  $x\in\BBV$ we set $U_x=x+\Int(\gamma)$.

\spa
(a)-(i) Denote by $\popb{\rho}$ the inverse image by $\rho$ in the category of presheaves. Let $F\in\md[\cor_\Vg]$ and let $x\in\BBV$. One has 
 $(\popb{\rho}\oim{\rho}F)(U_x)\simeq F(U_x)$. Since the family $\{U_x\}_{x\in \BBV}$ is  a basis of the topology of $\Vg$, we conclude that $\opb{\rho}\oim{\rho}F\simeq F$. 

 \spa
 (a)-(ii) Let $x\in\BBV$ and consider the skyscraper presheaf  on $\Xg$ at $x$, denoted here $\delta_x$. Then
  $\opb{\rho}\delta_x(x)\simeq0$  since 
for a sheaf $F$ on $\Vg$, $F(U_x)\simeq \sprolim F(U_y)$ where the limit ranges over the family $y\in U_x$. 

 \spa
 (b)-(i) Denote by $\popb{\beta}$ the inverse image by $\beta$ in the category of presheaves. Let $F\in\md[\cor_\Wg]$ and let $x\in\BBV$. One has 
 $(\popb{\beta}\oim{\beta}F)(x+\gamma)\simeq F(x+\gamma)$. Since the family $\{x+\gamma\}_{x\in\BBV}$ is a basis of the topology of $\Wg$, we conclude that 
$\opb{\beta}\oim{\beta}F\simeq F$.
 
 \spa
 (b)-(ii) Let $G\in\md[\cor_{\Xg}]$. Then $(\oim{\beta}\opb{\beta}G)(x)\simeq
(\opb{\beta}G)(x+\gamma)\simeq G(x)$.

\spa
(c) follows from (a) and (b) since $\oim{\rho}\simeq\oim{\beta}\circ\oim{\alpha}$.
\end{proof}
It follows from Proposition~\ref{pro:be} that one may consider the category of $\gamma$-sheaves as a full subcategory of that of persistent modules.

\subsection{The case of dimension one}\label{subsection:dimone}

The aim of this paper is to describe the category $\Derb_{\rcg}(\cor_\BBV)$. We first treat the case of the dimension one, where things are particularly simple. We shall study the category of $\gamma$-sheaves and construct a category of barcodes in dimension one, proving the equivalence of these categories in Theorem~\ref{th:B1}.
Note that various constructions of  categories of barcodes already exist in the literature. See in particular~\cites{BS14,BSD15,BL16}.

 We denote by $t$ a coordinate on $\R$ and by $(t;\tau)$ the associated homogeneous coordinates on $T^*\R$.
Therefore, $F\in \mdrc[\cor_\R]$ if there exists a discrete set $Z\subset\R$ such that $F$ is locally constant on $\R\setminus Z$ and moreover, the stalk of $F$ at each point of $\R$ is finite-dimensional.

In the sequel, an interval
means a non-empty convex  subset of $\R$.

\begin{proposition}
Let $F\in\mdrc[\cor_\R]$ and assume that $F$ has compact support.
Then, there exist a finite set $A$ and a family of intervals $\{I_\alpha\}_{\alpha \in A}$ 
such that $F \simeq \soplus_{\alpha \in A}\cor_{I_\alpha}$.
Moreover such a decomposition is unique. 
\end{proposition}
This result is due to Guillermou~\cite{Gu16}*{Cor.~7.3}  who deduced it from  Gabriel's theorem on the representation of quivers. We shall extend it to the non compact case, obtaining a result 
very similar to a theorem of Crawley-Boevey~\cite{CB14}*{Theorem 1.1}.

\begin{theorem}\label{th:Gth2}
Let $F\in\mdrc[\cor_\R]$.
Then, there exists a locally finite family of intervals $\{I_\alpha\}_{\alpha \in A}$ 
such that $F \simeq \soplus_{\alpha \in A}\cor_{I_\alpha}$.
Moreover such a decomposition is unique. 
\end{theorem}

\Proof
For $n\in\Z_{>0}$, set $U_n=]-n,n[$ and
$F_n=F\tens\cor_{U_n}$.
Then by the theorem above,
there exists a finite family $\{I^{(n)}_\alpha\}_{\alpha \in A_n}$
of intervals  in $U_n$ such that $F_n\simeq\oplus_{\al\in A_n}\cor_{I^{(n)}_\al}$.
Then there exists an injective map $A_n\monoto A_{n+1}$ 
(hereafter we identify $A_n$ as a subset of $A_{n+1}$ by this injective map) such that
$I^{(n)}_\al=I^{(n+1)}_\al\cap U_n$ for $\al\in A_n$.
Set $A=\bigcup_{n\in\Z_{>0}}A_n$,
Then, for any $\al\in A$, there exists a unique  interval $I_\al$
such that
$I_\al\cap U_n=I^{(n)}_\al$ for any $n$ such that $\al\in A_n$.
Then $\{I_\al\}_{\al\in A}$ is a locally finite family of intervals.
Note that $A_n=\set{\al\in A}{I_\al\cap U_n\not=\emptyset}$ and 
$I^{(n)}_\al=I_\al\cap U_n$ for $\al\in A_n$.

Let us show that $F$ is isomorphic to
$\soplus_{\al\in A}\cor_{I_\al}$.
For any $n$, there exists an isomorphism
$$\phi_n\cl F\vert_{U_n}\isoto
\soplus_{\al\in A}\cor_{I_\al}\vert_{U_n}
\simeq \soplus_{\al\in A_n}\cor_{I_\al}\vert_{U_n}.$$
The restriction map
$$\Hom[{\md[\cor_{U_m}]}]\bl\cor_{I_\al}\vert_{U_m},\cor_{I_\beta}\vert_{U_m}\br
\to
\Hom[{\md[\cor_{U_n}]}]\bl\cor_{I_\al}\vert_{U_n},\cor_{I_\beta}\vert_{U_n}\br
$$
is injective for $m\ge n$ and $\al,\beta\in A_n$.
Indeed, the injectivity follows from the commutative diagram
$$\xymatrix{
\mbox{$\rule[-2ex]{0ex}{3ex}
\Hom[{\md[\cor_{U_m}]}]\bl\cor_{I_\al}\vert_{U_m},\cor_{I_\beta}\vert_{U_m}\br$}
\ar[r]\ar@{>->}[d]&
{\Hom[{\md[\cor_{U_n}]}]\bl\cor_{I_\al}\vert_{U_n},
\cor_{I_\beta}\vert_{U_n}\br}\ar[d]\\
{\Hom[{\md[\cor_{I_\al\cap I_\beta\cap U_m}]}]
\bl\cor_{I_\al\cap I_\beta\cap U_m},\cor_{I_\al\cap I_\beta\cap U_m}\br\ }
\ar@{>->}[r]&
{\Hom[{\md[\cor_{I_\al\cap I_\beta\cap U_n}]}]
\bl\cor_{I_\al\cap I_\beta\cap U_n},\cor_{I_\al\cap I_\beta\cap U_n}\br}
}
$$
Here the injectivity of the bottom arrows follows from the fact that
$I_\al\cap I_\beta\cap U_m$ is empty if $I_\al\cap I_\beta\cap U_n$ is empty.

Hence the restriction map
\eq
\End_{\md[\cor_{U_m}]}\bl\soplus_{\al\in A_n}\cor_{I_\al}\vert_{U_m}\br
\To\End_{\md[\cor_{U_n}]}\bl\soplus_{\al\in A_n}\cor_{I_\al}\vert_{U_n}\br
\label{mor:mn}
\eneq
is injective for $m\ge n$.
Therefore, if
an automorphism $f$
of $\soplus_{\al\in A}\cor_{I_\al}\vert_{U_n}\simeq
\soplus_{\al\in A_n}\cor_{I_\al}\vert_{U_n}$, as well as $f^{-1}$, lifts
to an endomorphism of
$\soplus_{\al\in A_n}\cor_{I_\al}\vert_{U_m}$,
then it lifts to
an automorphism of
$\soplus_{\al\in A_n}\cor_{I_\al}\vert_{U_m}$.

By the injectivity of \eqref{mor:mn} and the fact that
$\dim\bigl(\End_{\md[\cor_{U_n}]}\bl\soplus_{\al\in A_n}\cor_{I_\al}\vert_{U_n}\br\bigr)
<\infty$, there exists
$m\ge n$ such that
$$\End_{\md[\cor_{U_k}]}\bl\soplus_{\al\in A_n}\cor_{I_\al}\vert_{U_k}\br
\To\End_{\md[\cor_{U_m}]}\bl\soplus_{\al\in A_n}\cor_{I_\al}\vert_{U_m}\br
$$
is an isomorphism  for any $k\ge m$.

Thus we conclude that
the image $K_{k,n}$ of the restriction map
$$\Aut\bl\soplus_{\al\in A}\cor_{I_\al}\vert_{U_k}\br
\To\Aut\bl\soplus_{\al\in A}\cor_{I_\al}\vert_{U_n})
$$
is equal to $K_{m,n}$ for any $k\ge m$.
Let $P_n$ be the set of isomorphisms
$F\vert_{U_n}\isoto \soplus_{\al\in A}\cor_{I_\al}\vert_{U_n}$.
Then $\{P_n\}_{n\in\Z_{>0}}$ is a projective system of non-empty sets.
Moreover,
for any $n$, there exists $m\ge n$ such that
$\Im(P_k\to P_n)=\Im(P_m\to P_n)$ for any $k\ge m$.
Set $\tilde P_n=\Im(P_m\to P_n)\subset P_n$.
Then $\{\tilde P_n\}_{n\in\Z_{>0}}$ is a projective system of
non-empty sets such that
the map $\tilde P_m\to\tilde P_n$ is surjective for any $m\ge n$.

Hence, by replacing 
$\phi_n$, 
we can choose $\phi_n\in\tilde P_n$ inductively so that we have
$\phi_{n+1}\vert_{U_{n}}=\phi_{n}$ for every $n$. 
Thus we conclude that $F\simeq\soplus_{\al\in A}\cor_{I_\al}$.
\QED

\begin{corollary}\label{cor:dimhom1}
Let $F,G\in \mdrc[\cor_\R]$. Then $\Ext{j}(F,G)=0$ for $j>1$.
\end{corollary}
\begin{proof}
By Theorem~\ref{th:Gth2}, we may assume that $F=\soplus_{a\in A}F_a$
 and $G=\prod_{b\in B}G_b$  with $F_a=\cor_{I_a}$ and $G=\cor_{J_b}$ where 
$I_a$ and $J_b$ are intervals. 
Since 
\eqn
&&\Ext{j}\bl\soplus_{a\in A} F_a,\prod_{b\in B}G_b\br\simeq\prod_{(a,b)\in A\times B}\Ext{j}(F_a,G_b),
\eneqn
we are reduced to prove the result with $F$ and $G$ replaced with $F_a$ and $G_b$ and in this case the result  is obvious. 
\end{proof}

\begin{example}
In Corollary~\ref{cor:dimhom1}, one cannot replace $j>1$ with $j\geq 1$. Indeed, 
one has $\Ext{1}(\cor_{[0,+\infty[},\cor_{]-\infty,0[})\simeq\cor$. (See Lemma~\ref{le:extn} for a generalization in higher dimension.)
\end{example}
Corollary~\ref{cor:dimhom1} classically  implies

\begin{corollary}\label{cor:dimhom2}
Let $F\in\Derb_{\Rc}(\cor_\R)$. Then $F\simeq\soplus_jH^j(F)\,[-j]$.
\end{corollary}
In other words, any object of the bounded derived category is isomorphic to the direct sum of its shifted cohomology objects. The same phenomena appears with the bounded derived  category of the category of  $\Z$-modules or that of $\cor[X]$-modules, for $\cor$ a field.

\subsubsection*{Barcodes}
We denote by $\R_\gamma$ the set $\R$ endowed with the $\gamma$-topology where $\gamma=]-\infty,0]$.
Hence, the $\gamma$-open sets are the open subsets  $\R_{<t}\eqdot]-\infty,t[$ and the  $\gamma$-closed sets are the closed subsets 
$\R_{\geq t}\eqdot [t,+\infty[$, with $t\in\ol\R$. 
The non-empty $\gamma$-locally closed sets are the intervals
$[a,b[$ with $-\infty\le a<b\le+\infty$.

\begin{definition}\label{def:barcode}
A {\em $\gamma$-barcode} $(A, I)$, or simply a barcode,  is the data of  a 
set of indices $A$ and a family $I=\{I_\alpha\}_{\alpha\in A}$ of intervals 
$I_\alpha=\intervCO{a_\alpha,b_\alpha}\subset\R$
with $-\infty\le a_\alpha<b_\alpha\le+\infty$,  these data satisfying
\eq\label{eq:bar1}
&&\parbox{75ex}{
the family $\{I_\alpha\}_{\alpha\in A}$ is locally finite on $\R$,  that is, 
for any compact subset $K$ of $\R$,
$\set{\al\in A}{I_\al\cap K\not=\emptyset}$ is finite.
}
\eneq
In particular, $A$ is countable.  
The support of the barcode $(A, I)$, denoted  by $\supp(A,I)$, is the closed set $\bigcup_{\alpha\in A}\ol{I_\alpha}$.
\end{definition}
Note that we do not ask $I_\alpha\neq I_\beta$ for  $\alpha\neq \beta$. Otherwise, we should have to endow each $I_\alpha$ with a multiplicity $m_\alpha\in\N$.

We shall now construct a category of $\gamma$-barcodes, with Theorem~\ref{th:B1} in view.

Let us say that a barcode $(A,I)$ is \define{elementary} if $A\simeq\rmpt$.  We shall identify the elementary barcode $(\rmpt,I)$ with the interval $I$.

Given two elementary barcodes $\intervCO{a,b}$ and $\intervCO{c,d}$, we set
\eq\label{eq:hombar1}
\Hom[\Barc](\intervCO{a,b},\intervCO{c,d})&=&\begin{cases}
\cor&\mbox{if $a\leq c<b\leq d$,}\\
0&\mbox{otherwise.}\end{cases}
\eneq
Note that 
\eq\label{eq:hombar5}
\Hom[\Barc](\intervCO{a,b},\intervCO{c,d})&\simeq &\Hom(\cor_{\intervCO{a,b}},\cor_{\intervCO{c,d}}).
\eneq
Given two barcodes $(A,I)$ and $(B,J)$ with  $I_{\alpha}=\intervCO{a_{\alpha},b_{\alpha}}$ and 
$J_\beta=\intervCO{c_{\beta},d_{\beta}}$, we set
\eq\label{eq:hombar2}
\Hom[\Barc]((A,I),(B,J))=\prod_{(\alpha,\beta)\in A\times B}\Hom[\Barc](\intervCO{a_\alpha,b_\alpha},\intervCO{c_\beta,d_\beta}).
\eneq
For $u\in \Hom[\Barc]((A,I),(B,J))$ and $v\in \Hom[\Barc]((B,J),(C,K))$, the composition $v\circ u$ 
is defined as follows.  Let  $u=\{c_{\alpha,\beta}\}_{(\alpha,\beta)\in A\times B}$ and 
$v=\{c_{\beta,\gamma}\}_{(\beta,\gamma)\in B\times C}$ 
with $c_{\alpha,\beta}, c_{\beta,\gamma}\in\cor$. 
One sets $v\circ u=\{c_{\alpha,\gamma}\}_{(\alpha,\gamma)\in A\times C}$ with 
\eq\label{eq:hombarcomp}
&&c_{\alpha,\gamma}=\sum_{\beta\in B}c_{\alpha,\beta}\cdot  c_{\beta,\gamma}
\qt{ if $\Hom[\Barc](I_\al, K_\gamma)=\cor$.}
\eneq
For a given $(\alpha,\gamma)$, the sum in~\eqref{eq:hombarcomp} is finite.
Indeed, given $\intervCO{a,b}$ and $\intervCO{e,f}$, consider the intervals $\intervCO{c_\beta,d_\beta}$
with $-\infty\leq a\leq c_\beta<b\leq d_\beta\leq+\infty$ and $-\infty\leq c_\beta\leq e < d_\beta\leq f\leq+\infty$. 
If $b<+\infty$ or $-\infty <e$ then either $b$ or $e$ belongs to $\intervCO{c_\beta,d_\beta}$ and the set of such $\beta$ is finite thanks to~\eqref{eq:bar1}.  Now assume that $b=+\infty$ and $-\infty=e$. Then $c_\beta=-\infty$ and $d_\beta=+\infty$ and again the family of such $\beta$ must be finite. 

\begin{notation}
We denote by   $\Barc$ the category constructed above and  call  it the category of $\gamma$-barcodes.
\end{notation}
\begin{remark}
The category $\Barc$ contains much more morphisms than the category constructed in~\cite{BL16}. 
\end{remark}
\begin{lemma}
The category $\Barc$ is additive.
\end{lemma}
\begin{proof}
(i) The $0$-barcode is $0=(A,I)$ with $A=\emptyset$. 

\spa
(ii) Given two barcodes  $(A,I)$  and $(B,J)$, we set $(A,I)\oplus(B,J)=(A\sqcup B, I\sqcup J)$. 
\end{proof}

Then we define a functor  $\Psi\cl \Barc\to \mdrcg[\cor_\R]$ as follows. We set
\eq\label{eq:fctB}
\Psi(A,I)&=&\soplus_{\alpha\in A}\cor_{I_\alpha}.
\eneq
By~\eqref{eq:hombar5}, both $\Hom[\Barc](\intervCO{a,b},\intervCO{c,d})$ and 
$\Hom(\cor_{\intervCO{a,b}},\cor_{\intervCO{c,d}})$ are simultaneously $\cor$ or $0$. 
Hence,  we define 
$\Psi\cl \Hom[\Barc](\intervCO{a,b},\intervCO{c,d})
\to\Hom(\cor_{\intervCO{a,b}},\cor_{\intervCO{c,d}})$ 
as the identity of $\cor$. 

Note that $\Psi$ commutes with the composition of morphisms. 

Then we extend $\Psi$ by linearity: 
\eqn
\Psi\;\cl\;\Hom[\Barc]((A,I),(B,J))&=&
\prod_{\alpha\in A,\beta\in B}
\Hom[\Barc](I_\alpha,J_\beta)\\
&\isoto& \prod_{\alpha\in A,\beta\in B}
\Hom\bl\cor_{I_\alpha}\cor_{J_\beta}\br\\
&\simeq& \Hom(\soplus_{\alpha\in A}\cor_{I_\alpha},\soplus_{\beta\in B}\cor_{J_\beta})\,.
\eneqn
Here, we have used 
the fact that   the sum is locally finite. 

This construction shows that $\Psi$ is additive and fully faithful.
\begin{theorem}\label{th:B1}
The functor $\Psi\cl \Barc\to \mdrcg[\cor_\R]$  is an equivalence of additive categories. 
\end{theorem}
\begin{proof}
By Theorem~\ref{th:Gth2},  there exist a locally finite family of intervals $\{I_\alpha\}_{\alpha \in A}$ 
such that $F \simeq \soplus_{\alpha \in A}\cor_{I_\alpha}$. 
Since $\musupp(\cor_{I_\alpha})\subset \R\times\gammac$, the intervals $I_\alpha$ are of the type
$[a,b[$ with $-\infty\leq a<b\leq+\infty$. 
\end{proof}

Note that
\eq\label{eq:supppsi}
&&\supp(\Psi(A,I))=\supp(A,I)\qt{for a barcode $(A,I)$.}
\eneq

Let us denote by  $\opb{\Psi}$ a   quasi-inverse of the functor $\Psi$. 
\begin{definition}\label{def:barcfct}
The barcode functor $\mdrcg[\cor_{\R}]\isoto\Barc$ is the functor $\opb{\Psi}$. 
\end{definition}

Definition~\ref{def:barcfct} extends to the derived category. Denote by   $ H^\scbul$ the functor  $F\mapsto\soplus_jH^j(F)$ from $\Derb(\cor_\R)$ to $(\md[\cor_\R])^{(\Z)}$. One still calls 
\eqn
&&\opb{\Psi}\circ H^\scbul\cl\Derb_{\rcg}(\cor_{\R})\to\Barc^{(\Z)}
\eneqn
 the barcode functor.

\begin{example}\label{exa:persi3}
Let  $M$ be a real analytic manifold  endowed with a subanalytic distance $d$, let $K\subset M$ be a compact subanalytic subset and let $f(x)=d(x,K)$. Then 
\eqn
&&\Gamma^+_f=\{(x,t);d(x,K)\leq t\}
\eneqn
and $f$ satisfies  hypothesis~\ref{hyp:persi1}.  The barcodes of $K$ are those of $\roim{f}\cor_{\Gamma^+_f}$, that is, the image of $\roim{f}\cor_{\Gamma^+_f}$ by the functor $\opb{\Psi}\circ H^\scbul$ of Definition~\ref{def:barcfct}.
\end{example}

A natural question is to extend the definition of  barcodes to $\R^n$ when $n>1$. For that purpose, it is natural to replace the cone $\{t\leq0\}$ of $\R$ with a closed convex proper cone $\gamma$ of $\BBV$ as in subsection~\ref{subsection:gtop}.

\section{Sheaves on vector spaces}\label{section:dist}
In this section, we denote by $\BBV$ a real vector space of finite dimension $n$.

We endow $\BBV$ with a  Euclidean structure and denote  by $\vvert\cdot\vvert$ the norm on $\BBV$. 
Let us denote by $B_a$ the closed ball with center  the origin
and radius  $a\ge0$ :
\eqn
&&B_a\seteq\set{x\in\BBV}{\Vert x\Vert\le a}. 
\eneqn

\subsection{Convolution}
References for this subsection are made to~\cites{Ta08,GS14}.

Consider the maps 
\eqn
&&s\cl\BBV\times\BBV\to\BBV,\quad s(x,y)=x+y,\\
&&q_i\cl\BBV\times\BBV\to\BBV\ (i=1,2), \quad q_1(x,y)=x,\ q_2(x,y)=y.
\eneqn
Recall that $a\cl\BBV\to\BBV$ denotes the antipodal map, $x\mapsto -x$. For a sheaf $F$, we  set 
for short $F^a=\opb{a}F$. 

One defines the convolution functor $\star\cl \Derb(\cor_\BBV)\times \Derb(\cor_\BBV)\to\Derb(\cor_\BBV)$ and its adjoint 
$\rhomc\cl (\Derb(\cor_\BBV))^\rop\times \Derb(\cor_\BBV)\to\Derb(\cor_\BBV)$ by the formulas:
\eqn
F\star G&\eqdot&\reim{s}(F\etens G),\\
\rhomc(F,G)&\eqdot& \roim{q_1}\rhom(\opb{q_2}G, \epb{s}F)\\
&\simeq&\roim{s}\rhom(\opb{q_2}(G^a), \epb{q_1}F).
\eneqn
For $F_1,F_2,F_3\in\Derb(\cor_{\BBV})$, we have
\eq\label{eq:adjstarhomc}
&&\RHom(F_1\star F_2,F_3)\simeq\RHom(F_1,\rhomc(F_2,F_3)).
\eneq

With $\star$ as a tensor product,
$\Derb(\cor_\BBV)$ is a commutative monoidal category with a unit object $\cor_{\{0\}}$:
$\cor_{\{0\}}\star F\simeq F$ functorially in $F\in\Derb(\cor_{\BBV})$.

The functors $\star$ and $\rhomc$ induce functors (see~\cite{GS14}*{Cor.~3.1.4})
\eqn
\scbul\star\scbul&\cl& \Derb(\cor_\BBV)\times \Derb_{\gammac}(\cor_\BBV)\to\Derb_{\gammac}(\cor_\BBV),\\[1ex]
\rhomc(\scbul,\scbul)&\cl& (\Derb(\cor_\BBV))^\rop\times \Derb_{\gammac}(\cor_\BBV)\to\Derb_{\gammac}(\cor_\BBV),\\
&\cl& (\Derb_{\gammac}(\cor_\BBV))^\rop\times \Derb(\cor_\BBV)\to\Derb_{\gammac}(\cor_\BBV).
\eneqn

We set $X=\R\times\BBV$ and denote by $t$ the coordinate on $\R$. 
Following~\cite{GKS12}*{Exa~3.10} we recall that there exists $K\in\Derb_\Rc(\cor_X)$ unique up to  isomorphism such that
\eqn
&&\musupp(K)\subset \{(t,x;\tau,\xi);\tau=\vvert \xi\vvert\not=0,\;
 x=-t\tau^{-1}\xi\}\cup\{\tau=\xi=0\},\\
&&K\vert_{t=0}\simeq\cor_{\{0\}}.
\eneqn
Moreover, there is a distinguished triangle
\eqn
&&\cor_{\{ t<-\Vert x\Vert\}}\,[n]\to K\to 
\cor_{\{\Vert x\Vert\leq t\}}\to[+1].
\eneqn
Set $K_a=K\vert_{t=a}$.  Hence 
\eq\label{eq:Ka}
K_a\simeq 
\begin{cases}\cor_{\{\Vert x\Vert\leq a\}}&\text{for $a\geq0$,}\\
\cor_{\{\Vert x\Vert<-a\}}[n]&\text{for $a<0$.}\end{cases}
\eneq
We can easily check the isomorphism
\eq\label{eq:KaKb}
&&K_a\star K_b\simeq K_{a+b}\mbox{ for }a,b\in\R.
\eneq
Hence, $K_a\star$ is an auto-equivalence of the category $\Derb(\cor_\BBV)$ as well as of
the category $\Derb_{\rcg}(\cor_\BBV)$. A quasi-inverse is given by $K_{-a}\star$. 

Note that 
\eq
&&\Hom(K_a\star F,G)\simeq \Hom(F,K_{-a}\star G)
\eneq
and thus
\eqn
&&K_a\star F\simeq\rhomc(K_{-a},F).
\eneqn
For $c\geq0$ we have a canonical morphism
$K_c\to\cor_{\{0\}}$, which induces 
canonical morphisms  for $a\geq b$:
\eqn
\chi_{b,a}&\cl& K_a\To K_b,\\
\chi_{b,a}\star F&\cl&
K_a\star F\to K_b\star F.
\eneqn
 In particular, one has
\eq
&&K_a\star F\To[\;\chi_{0,a}\star F\;] F\To[\;\chi_{b,0}\star F\;] K_b\star F\quad\text{for $a\ge0\ge b$.}
\eneq

Moreover, one has  (recall that $\RD_\BBV$ denotes the duality functor, 
see \eqref{eq:dualfct}) :
\eq
&&\RD_\BBV(K_a)\simeq K_{-a}.
\eneq

\begin{lemma}\label{le:dualconv}
For $F\in\Derb(\cor_\BBV)$, one has
\eq
&&\RD_\BBV(K_a\star F)\simeq K_{-a}\star\RD_\BBV(F).
\eneq
\end{lemma}
\begin{proof}
One has 
\eqn
\RD_\BBV(K_a\star F)&\simeq&\rhom(\reim{s}(K_a\etens F),\omega_\BBV)\\
&\simeq&\roim{s}\rhom(K_a\etens F,\epb{s}\omega_\BBV)\\
&\simeq&\roim{s}\rhom(K_a\etens F,\omega_{\BBV\times\BBV})\\
&\simeq&\roim{s}(\RD_\BBV(K_a)\etens\RD_\BBV F)\simeq K_{-a}\star \RD_\BBV F.
\eneqn
Note that the last isomorphism follows from the fact that $K_a$ has a compact support.
\end{proof}

\subsection{Distance and the stability theorem}
A distance on persistence modules was first introduced  in~\cite{CCGGO09} (see also \cite{CSGO16}), generalized to higher dimension in~\cites{Le15,LW15} and adapted to sheaves in~\cite{Cu13}. This distance is based on an idea of thickening the open sets. Note that the idea of thickening constructible functions already appeared in~\cite{Sc91}, a paper inspired by~\cite{GRS83}. However, the idea of a distance in this framework is new as well as the so-called stability results, meaning a control of the distance after direct images. 

Here we also introduce a thickening of sheaves and a pseudo-distance for sheaves and we prove a stability result. Our construction is  based on the convolution of sheaves {\em in the derived setting} and therefore our results are essentially of a different nature from those mentioned above.

\begin{definition}\label{def:convdist}
Let $F, G \in \Derb(\cor_{\BBV})$ and let $a\geq0$.
One says that $F$ and $G$ are {\em $a$-isomorphic} if there are morphisms 
$f\cl K_{a}\star F\to G$ and $g\cl K_a\star G\to F$ 
which satisfies the following compatibility conditions:
the composition 
$K_{2a}\star F\To[K_a\star f] K_a\star G\To[g] F$
 coincides with  the natural morphism $\chi_{0,2a}\star F\cl
K_{2a}\star F\to F$ and the composition 
$K_{2a}\star G\To[K_a\star g] K_a\star F\To[f] G$
coincides with  the natural morphism $\chi_{0,2a}\star G\cl K_{2a}\star G\to G$. 
 \end{definition}

Note that if $F$ and $G$ are $a$-isomorphic, then
they are $b$-isomorphic for any $b\ge a$. 

One sets 
\eqn&&\dist(F,G)=\inf\Bigl(\{+\infty\}\cup\{a\in\R_{\geq0}\,;\,
\text{$F$ and $G$ are $a$-isomorphic}\}\Bigr)
\eneqn
and calls $\dist(\scbul,\scbul)$ the \define{convolution distance}.

Note that for $F,G,H\in \Derb(\cor_{\BBV})$, 
 \begin{itemize}
 \item $F$ and $G$ are $0$-isomorphic
if and only if $F\simeq G$,
 \item
 $\dist(F,G)=\dist(G,F)$, 
 \item
 $\dist(F,G)\leq \dist(F,H)+\dist(H,G)$.
 \end{itemize}

\Rem\label{rem:notdist}
We don't know if $F$ and $G$ are $a$-isomorphic when $\dist(F,G)\le a$.
\enrem

\begin{example}\label{ex:aisom}
Let $F\in\Derb(\cor_\BBV)$ and assume that $\supp(F)\subset B_a$. 
Set $L\eqdot \rsect(\BBV;F)$.
Recall that, for a closed subset $S$ of $\BBV$, one denotes by
$L_S\in\Derb(\cor_\BBV)$ the constant sheaf with stalk $L$ on $S$ extended by $0$ on $\BBV\setminus S$.  Let $G\eqdot L_{\{0\}}$. 
We have
\eqn
&&K_a\star G\simeq L_{B_a}.
\eneqn
Denote by $q\cl\BBV\to\rmpt$ the unique map from $\BBV$ to $\rmpt$. The morphism $\opb{q}\roim{q}F\to F$ defines the map
$L_\BBV\to F$ and $F$ being supported by $B_a$, we get the morphism 
$g\cl K_a\star G\simeq  L_{B_a}\to F$. 
On the other hand,we have $(K_a\star F)_0\simeq L$ which defines  $f\cl K_a\star F\to G$.
One easily checks that $f$ and $g$ satisfy the compatibility conditions in Definition~\ref{def:convdist}. Therefore
\eqn
&&\dist(F,L_{\{0\}})\leq a.
\eneqn
In particular, a non-zero object can be $a$-isomorphic to
the zero object. 
\end{example}

\Rem
\bnum
\item
If $\dist(F,G)<+\infty$, then we have
$$\rsect(\BBV; F)\simeq\rsect(\BBV; G)
\qtq\rsectc(\BBV; F)\simeq\rsectc(\BBV; G).$$
 This follows from
$\rsect(\BBV; K_a\star F)\isoto\rsect(\BBV; F)$
and $\rsectc(\BBV; K_a\star F)\isoto\rsectc(\BBV; F)$
 for $a\ge0$. 

Together with Example~\ref{ex:aisom}, for $F,G\in \Derb(\cor_\BBV)$
with $\supp(F),\;\supp(G)\subset B_a$, 
the condition $\dist(F,G)\le 2a$ holds if and only if
$\rsect(\BBV; F)\simeq\rsect(\BBV; G)$.

\item 
Let $A$ and $B$ be closed convex subsets of $\BBV$.
Then $\dist(\cor_A,\cor_B)\le a$ if and only if
$A\subset B+B_a$ and $B\subset A+B_a$.

Indeed, we may assume that $A$ and $B$ are non-empty.
If   $\cor_A$ and $\cor_B$ are $a$-isomorphic,
then there exists a non-zero morphism
$\cor_{A+B_a}\simeq K_a\star\cor_A\to\cor_B$.
Hence, we have $A+B_a\supset B$.
Similarly we have $B+B_a\supset A$.
The converse implication can be proved similarly. \label{remitem:2}
\item
Let $\gamma$ and $\gamma'$ be closed convex cones.
If $\dist(\cor_\gamma,\cor_{\gamma'})<+\infty$,
then one has $\gamma=\gamma'$.
This follows from \eqref{remitem:2}. 

\item 
Let $U$ and $V$ be open convex subsets of $\BBV$.
Then $\dist(\cor_U,\cor_V)\le a$ if and only if
$U\subset V+B_a$ and $V\subset U+B_a$.

This follows from \eqref{remitem:2},
$\RD_\BBV(\cor_U[n])\simeq\cor_{\ol{U}}$,
$\RD_\BBV(\cor_{\ol{U}})\simeq\cor_U[n]$
and Proposition~\ref{prop:distance}~\eqref{item:distance}
below. 
\item
Let $a\in\R$ and recall that the object $K_a$ satisfies~\eqref{eq:Ka}. Using the isomorphism~\eqref{eq:KaKb}, one easily checks that for $a\geq0$ one has
\eqn
\dist(K_{a},K_0)\leq a ,
 &&\dist(K_{-a},K_0)\leq a.
 \eneqn
 Since $K_0\simeq\cor_{\{0\}}$ and $K_{-a}$ is the $n$-shifted constant sheaf on the open ball with radius $a$, we get an example of two sheaves concentrated in different degrees whose distance is finite.
\ee
\end{remark}

\Prop\label{prop:distance}
\bnum
\item
If $\dist(F,G)\le a$, then
$\dist\bl\RD_\BBV(F),\RD_\BBV(G)\br\le a$.\label{item:distance}
\item
Assume that $\dist(F_i,G_i)\le a_i$ $(i=1,2)$.
Then one has
$$\dist(F_1\star F_2, G_1\star G_2)\le a_1+a_2\qtq
\dist\bl\rhomc(F_1, F_2),\rhomc(G_1,G_2)\br\le a_1+a_2.$$
\item
$\dist(\opb{\phig}\roim{\phig}F,\opb{\phig}\roim{\phig}G)\le
\dist(F,G)$.\label{item:3}
\ee
\enprop
\begin{proof}
(i) follows from Lemma~\ref{le:dualconv}.

\spa
(ii) Let $f_i\cl K_{a_i}\star F_i\to G_i$ ($i=1,2$). We get a map
\eqn
&&f_1\star f_2\cl K_{a_1}\star F_1\star K_{a_2}\star F_2\to G_1\star G_2
\eneqn
and $K_{a_1}\star F_1\star K_{a_2}\star F_2\simeq K_{a_1+a_2}\star(F_1\star F_2)$. The end of the proof is straightforward and the case of $\rhomc$ is similar.

\spa
(iii) follows from
$\opb{\phig}\roim{\phig}F\simeq
\roim{s}(\cor_{\gamma^a}\etens F)\simeq \rhomc(\cor_{\Int(\gamma)}[n],F)$ and (ii).
\end{proof}

For a set $X$ and a map $f\cl X\to\BBV$, one sets
\eqn
&&\vvert f\vvert=\sup_{x\in X}\vvert f(x)\vvert.
\eneqn

 \begin{theorem}[{The stability theorem}]\label{th:stab}
 Let $X$ be a locally compact space and  let
$f_1,f_2\cl X\to\BBV$ be two continuous maps.
Then,  for any $F\in\Derb(\cor_X)$, we have
$$\dist(\roim{f_1}F,\roim{f_2}F)\leq \vvert f_1-f_2\vvert\qtq \dist(\reim{f_1}F,\reim{f_2}F)\leq\vvert f_1-f_2\vvert.$$
\end{theorem}
\Proof
(i) Since the proofs for $\roim{f}$ and $\reim{f}$ are similar, we shall only treat the case of $\roim{f}$.

\spa
(ii) Set $a=\vvert f_1-f_2\vvert$. Let us show that 
$\roim{f_1}F$ and $\roim{f_2}F$ are $a$-isomorphic. 
Let us introduce some notations. 
Let $p_1\cl X\times\BBV\to X$ and 
$p_2\cl X\times\BBV\to \BBV$ be the projections.
We set
\eqn
G_i&=&\set{(x,f_i(x))\in X\times\BBV}{x\in X},\\
S_i&=&\set{(x,y)\in X\times\BBV}{\Vert y-f_i(x)\Vert\le a},\\
S'_i&=&\set{(x,y)\in X\times\BBV}{\Vert y-f_i(x)\Vert\le 2a}.
\eneqn for $i=1,2$.
Then we have
\eqn
\roim{f_i}F&\simeq&\roim{p_2}\bl\cor_{G_i}\tens \opb{p_1}F\br,\\
K_a\star\roim{f_i}F&\simeq&\roim{p_2}\bl\cor_{S_i}\tens \opb{p_1}F\br,\\
K_{2a}\star\roim{f_i}F&\simeq&\roim{p_2}\bl\cor_{S'_i}\tens \opb{p_1}F\br.
\eneqn
Then $G_2\subset S_1$ induces a morphism $\cor_{S_1}\to\cor_{G_2}$, which induces
$K_a\star\roim{f_1}F\to \roim{f_2}F$.
Similarly, we construct
$K_a\star\roim{f_2}F\to \roim{f_1}F$.
Then, the composition
$$K_{2a}\star\roim{f_1}F\to K_a\star\roim{f_2}F\to \roim{f_1}F$$
is the canonical one, since these morphisms are induced by
$$\cor_{S'_1}\to\cor_{S_2}\to\cor_{G_1}.$$
The same argument holds when exchanging $f_1$ and $f_2$.
\QED

\subsection{\PL-sheaves}\label{subsection:PL}
\begin{definition}\label{def:polytope}
A {\em convex polytope} $P$ in $\BBV$ is 
the intersection of a finite family of open or closed affine half-spaces. 
\end{definition}
From now on, we shall also assume that  $\gamma$ is  polyhedral, that is, 
 \eq\label{hyp1b}
&&\parbox{75ex}{\em{$\gamma$ is a  closed convex proper polyhedral cone with non-empty interior. }
}\eneq

\begin{definition}\label{def:Flinear}
One says that $F\in \Derb_{\Rc}(\cor_\BBV)$ is   \PL\ (piecewise linear) 
if there exists a locally finite family $\{P_a\}_{a\in A}$ of
convex polytopes such that  $\BBV=\bigcup_{a\in A}P_a$ and 
$F\vert_{P_a}$ is constant for any $a\in A$.
\end{definition}

We shall use the notations (for a cone $\gamma$ satisfying~\eqref{hyp1b}):
\eq\label{not:2}
&&\left\{\begin{array}{l}
\Derb_{\rcl}(\cor_{\BBV})\eqdot\{F\in\Derb_\Rc(\cor_{\BBV}) ;\mbox{ $F$ is \PL}\},\\[1ex]
\Derb_{\rclg}(\cor_{\BBV})\eqdot\Derb_{\rcl}(\cor_{\BBV})\cap\Derb_{\gammac}(\cor_{\BBV}),\\[1ex]
\mdrcl[\cor_{\BBV}]\eqdot\mdrc[\cor_{\BBV}]\cap \Derb_{\rcl}(\cor_\BBV),\\[1ex]
\mdrclg[\cor_{\BBV}]\eqdot\mdrcl[\cor_\BBV]\cap\mdg[\cor_{\BBV}].
\end{array}\right.
\eneq

Since the following  result  is easy to prove, we omit the proof.
\begin{theorem}\label{th:PL}
 The category $\Derb_{\rcl}(\cor_\BBV)$ is triangulated. Moreover:
\bnum
\item If $F$ and $F'$ are \PL, then so are $F\tens F'$ and $\rhom(F,F')$.
\item Let $f\cl \BBV\to\BBV'$ be a linear map.
\bna
\item If $F'$ is a PL\ sheaf on $\BBV'$, then $\opb{f}F'$
is a \PL\ sheaf on $\BBV$.
\item If $F$ is a \PL\ sheaf on $\BBV$ and $\supp(F)$ is proper over $\BBV'$,
then $\roim{f}F$ is a \PL\ sheaf on $\BBV'$.
\ee
\item 
Assume~\eqref{hyp1b}. If  $F$ is 
\PL\ and $\supp(F)$ is $\gamma$-proper \lp see {\rm Definition~\ref{def:loclo}}\rp,
then $\opb{\phig}\roim{\phig}F$ is \PL.
\enum
\end{theorem}

\subsection{The approximation theorem}

\begin{theorem}[{The approximation theorem}]\label{th:approx}
Let  $F\in\Derb_{\Rc}(\cor_\BBV)$. For each $\epsilon>0$ there exists 
$G\in \Derb_{\rcl}(\cor_\BBV)$ such that $\dist(F,G)\leq\epsilon$
and $\supp(G)\subset \supp(F)+B_\eps$. 
\end{theorem}
\begin{proof}
We shall follow the notations of~\cite{KS90}*{\S\,8.1}. Recall that a simplicial complex $\bS=(S,\Delta)$ is the data consisting of a set $S$ and a set $\Delta$ of subsets of $S$, satisfying  certain  conditions (see loc.\ cit.\ Def.~8.1.1). For $\sigma\in\Delta$, one sets
\eqn
&&\vert\sigma\vert=\Bigl\{(x(p))_{p\in S}\in\R_{\geq0}^S\;;\;
\sum_{p\in S}x(p)=1,\;
\text{$x(p)=0$ for $p\notin \sigma$ and
$x(p)>0$ for $p\in\sigma$}
\Bigr\}\,.
\eneqn
Note that the $\vert\sigma\vert$'s are disjoint to each other. One also sets
\eqn
&&\vert\bS\vert=\bigcup_{\sigma\in\Delta}\vert\sigma\vert.
\eneqn
We endow $\vert\bS\vert$ with the induced topology from $\R_{\ge0}^S$.

There exist a simplicial complex $\bS=(S,\Delta)$ and a homeomorphism
$f\cl\vert\bS\vert\isoto\BBV$ such that 
\eqn
&&\text{$(\opb{f}F)\vert_{\vert\sigma\vert}$ is constant for any  $\sigma\in\Delta$}. 
\eneqn
Replacing $\bS$ with its successive barycentric 
subdivisions, we may assume
 further  that $\vvert f(x)-f(y)\vvert\leq\epsilon$ for any $\sigma\in\Delta$ and
$x,y\in\vert\sigma\vert$. Then we set
\eqn
&&g(x)=\sum_{p\in S}x(p)f(p).
\eneqn
Here $S$ is identified with a subset of $\vert\bS\vert$
by  $S\ni q\mapsto x(p)=\delta_{p,q}\in\R^S$.

The map $g\cl \vert\bS\vert\to\BBV$ is  piecewise  linear and continuous and satisfies:
\eqn
&&\vvert f(x)-g(x)\vvert\leq\epsilon.
\eneqn
Hence $g$ is a proper map.
Now we set $G=\roim{g}\opb{f}F$. Then 
$\supp(G)\subset \supp(F)+B_\eps$  and $G\in \Derb_{\rcl}(\cor_\BBV)$, and since $F\simeq \roim{f}\opb{f}F$, we have  by Theorem~\ref{th:stab}
\eqn
&&\quad  \dist(F,G)\leq\epsilon.
\eneqn
\end{proof}

One can approximate any $\gamma$-sheaf with a $\PL$-$\gamma$-sheaf.  Indeed:

\begin{corollary}\label{cor:approx}
Assume~\eqref{hyp1b}.
Let  $F\in\Derb_{\rcg}(\cor_\BBV)$ such that $\supp(F)$ is $\gamma$-proper. 
For each $\epsilon>0$ there exists 
$G\in \Derb_{\rclg}(\cor_\BBV)$  such that $\dist(F,G)\leq\epsilon$.
\end{corollary}
\begin{proof}
We shall apply Theorem~\ref{th:approx}.
There exists $G\in \Derb_{\rcl}(\cor_\BBV)$ such that
$\dist(F,G)\leq\epsilon$ and $\supp(G)\subset\Supp(F)+B_\eps$.
Hence $\supp(G)$ is $\gamma$-proper.
Then, $\opb{\phig}\roim{\phig}G\in \Derb_{\rclg}(\cor_\BBV)$ 
by Theorem~\ref{th:PL}. On the other hand,
Proposition~\ref{prop:distance}~\eqref{item:3} implies that
$$\dist\bl\opb{\phig}\roim{\phig}F,\opb{\phig}\roim{\phig}G\br
\le\dist(F,G)\le\epsilon.$$
Since $F$ is a $\gamma$-sheaf, one has
$\opb{\phig}\roim{\phig}F\simeq F$.
\end{proof}

\subsection{Barcodes (multi-dimensional case)}

Recall that 
a family of subsets  $Z=\{Z_\alpha\}_{\alpha\in A}$ is locally finite if for any compact subset  $K$ of $\BBV$, 
the set $\set{\alpha\in A}{Z_\alpha\cap K\neq\varnothing}$ is finite. 

\begin{definition}\label{def:bar}
Assume~\eqref{hyp1b}.
A {\em $\gamma$-barcode} $(A, Z)$ in $\BBV$, or simply, a barcode, 
is the data of  a set of indices  $A$ and a  family 
$Z=\{Z_\alpha\}_{\alpha\in A}$ of  subsets of $\BBV$, these data satisfying
\eq\label{eq:bar2}
&&\left\{\parbox{70ex}{
(i) the family $Z=\{Z_\alpha\}_{\alpha\in A}$ is locally finite in $\BBV$,\\[.8ex]
(ii)  the  $Z_\alpha$'s are  non-empty,  $\gamma$-locally closed, 
convex polytopes. 
}\right.
\eneq
The support of the  $\gamma$-barcode  $(A, Z)$, denoted  by  $\supp(A,Z)$, is the set $\bigcup_{\alpha\in A}\ol{Z_\alpha}$. 
\end{definition}
 Let us say that a barcode $(A,Z)$ is \define{elementary} if $A\simeq\rmpt$.  We shall identify the elementary barcode $(\rmpt,Z)$ with the  $\gamma$-locally closed  convex polytope  $Z$.

The barcodes are the objects of the additive category $\Barc$ that we shall describe now. 
\begin{itemize}
\item
the zero-barcode is $0=(A,Z)$ with $A=\emptyset$,
\item 
for two barcodes $(A,S)$  and $(B,Z)$, we set $(A,S)\oplus(B,Z)=(A\sqcup B, S\sqcup Z)$. In other words, 
the sum of $\{S_\alpha\}_{\alpha\in A}$ and $\{Z_\beta\}_{\beta\in B}$ is the barcode
$\{W_\gamma\}_{\gamma\in C}$ with $C=A\sqcup B$ and $W_\gamma=S_\alpha$ or $W_\gamma=Z_\beta$ according  as
 $\gamma=\alpha\in A$ or $\gamma=\beta\in B$. 
\item
For two elementary barcodes $S$ and $T$ one sets
\eq\label{eq:homST1}
&&\hs{3.5ex}\Hom[\Barc](S,T)=\begin{cases}
\cor&\mbox{if $S\cap T$ is non-empty, closed in $S$ and open in $T$},\\
0&\mbox{otherwise.}\end{cases}
\eneq
Note that
\eq\label{eq:homST2}
\Hom[\Barc](S,T)\simeq\Hom(\cor_S,\cor_T).
\eneq
\item
One extends this construction to barcodes by linearity. 
For two barcodes $(A,S)=\{S_\alpha\}_{\alpha\in A}$ and $(B,T)=\{T_\beta\}_{\beta\in B}$,
one sets
\eq\label{eq:hombarc1}
\Hom[\Barc]((A,S),(B,T))&=&
\prod_{\alpha\in A,\beta\in B}\Hom[\Barc](S_\alpha,T_\beta)\nonumber\\
&\simeq&\Hom(\soplus_{\alpha\in A}\cor_{S_\alpha},\prod_{\beta\in B}\cor_{T_\beta})\nonumber\\
&\simeq&\Hom(\soplus_{\alpha\in A}\cor_{S_\alpha},\soplus_{\beta\in B}\cor_{T_\beta}).
\eneq
(The last isomorphism  follows from the fact that the sum is locally finite, similarly  to the one-dimension case.)
The composition of $u\in \Hom[\Barc]((A,S),(B,T))$ and $v\in \Hom[\Barc](( B,T),(C,V))$ 
is defined similarly  to the one-dimension case (see~\eqref{eq:hombarcomp}), or by using~\eqref{eq:hombarc1}.
\end{itemize}
There is a natural functor 
\eq\label{eq:fctpsi}
&&\Psi\cl \Barc\to\mdrclg[\cor_\BBV],\quad
Z=\{Z_\alpha\}_{\alpha\in A}\mapsto \soplus_{\alpha\in A}\cor_{Z_\alpha}
\eneq
and the  functor $\Psi$ is fully faithful in view of~\eqref{eq:hombarc1}.
We have seen that when $\dim\BBV=1$ the functor $\Psi$ is an equivalence of categories. However, 
the functor  $\Psi$ is no more essentially surjective when  $\dim\BBV>1$
as seen in the following examples. 

\begin{example}
Denote by $(x,y)$ the coordinates on $\R^2$ and consider the sets 
\eqn
&&\gamma=\set{(x,y)}{x\leq0,y\leq0}, \\
&&A=(1,0)+\gamma^a, \quad B=(0,1)+\gamma^a, \quad C=(2,2)+\gamma^a.
\eneqn
Define the $\gamma$-sheaf $F$ by the exact sequence $0\to F\to\cor_A\oplus\cor_B\to\cor_C\to0$. Then 
$F\simeq \cor_A\oplus\cor_B$ on $\R^2\setminus C$ and $F$ has rank one on $\Int(C)$.

Let us show that $\End_{\mdrc[\cor_{\BBV}]}(F)\simeq\cor$.

\spa
Let $u$ and $v$ be the generators of $\sect(\BBV;\cor_A)$
and $\sect(\BBV;\cor_B)$, respectively.
Then one has $F(\BBV)=\cor(u-v)\subset\sect(\BBV;\cor_A\oplus\cor_B)$.
We have an exact sequence
$$0\To \cor_{A\setminus B}\to \cor _{A\cup B}\oplus \cor_{A\setminus C}
\To F\to 0.$$
Here, the composition
$\cor_{A\setminus C}\to F\to\cor_B$ vanishes
and the morphism
$ \cor _{A\cup B}\to F$ is given by $u-v$. 
Let $f\in\End(F)$.
We shall show $f\in\cor\id_F$.
We may assume that $f(u-v)=0$ from the beginning, i.e., the composition
$g\cl\cor _{A\cup B}\to F\To[f]F$ vanishes.
Then the composition $\cor_{A\setminus C}\to F\To[f]F\to \cor_A$ vanishes,
since it coincides with $g$ on $A\setminus B$. 
The composition  $\cor_{A\setminus C}\to F\To[f]F\to \cor_B$ vanishes
since $\Hom(\cor_{A\setminus C},\cor_B)\simeq0$. 
Hence $f=0$.

\spa

Therefore $F$ is indecomposable. 
Hence,  $F$ is not in the essential image of $\Psi$.
\end{example}
\begin{example}
Denote by $(x,y,z)$ the coordinates on $\R^3$ and consider the sets 
\eqn
&&\gamma=\set{(x,y,z)}{x\leq-(\vert y\vert+\vert z\vert)}, 
\quad S=\{(x,y,z);x=0, |y|+|z|=1\},\\
&&Z =(S+\gamma^a)\cap\{x<1\}.
\eneqn
Then $Z$ is  a $\gamma$-barcode. 
Since $Z\cap\set{(x,y,z)}{y=z=0}=\emptyset$,
 we consider the map
$r\cl Z\to \R^3\setminus\R\times\{(0,0)\}\to \R^2\setminus\{(0,0)\}\to\BBS^1$.
Then the composition $S\to Z\to[r]\BBS^1$ is an isomorphism.
 Let $L$ be a locally constant but non constant sheaf of rank $1$ on $\BBS^1$. 
Then, the sheaf $F=\opb{r}L$ is locally isomorphic to 
$\cor_Z$ but $F$ is not  isomorphic to $\cor_Z$ since  
$\oim{r}(F\tens\cor_S)\simeq L\not\simeq\cor_{\BBS^1}$. 
Hence $F$ is not in the essential image of $\Psi$.
\end{example}

Note that a different definition of higher dimensional barcodes is proposed in~\cite{Cu13}*{Def.~8.3.4}.

\begin{definition}\label{def:barcodesheaf}
An object of $\mdrclg[\cor_\BBV]$ is a barcode $\gamma$-sheaf if it is in the essential image of $\Psi$.
\end{definition}
%%%%%%%%%%%%%%%

\section{Piecewise linear $\gamma$-sheaves}\label{section:PL}

The aim of this section is to prove Theorem~\ref{th:linearcase}.

\subsection{Complements on the $\gamma$-topology}\label{subsection:gtop2}
We denote by $\gamma$ a  closed proper convex cone with non-empty interior, as in~\eqref{hyp1},
and recall  Definition~\ref{def:loclo}.

\Lemma\label{lem:Ugamma}\hfill
\bnum
\item
for any subset $A$ of $\BBV$, one has
$A+\Int(\gamma)=\ol A+\Int(\gamma)$.\label{item:1}

\item If $U$ is an open subset of $\BBV$, then we have
$U+\Int(\gamma)=U+\gamma$. In particular $U\subset U+\Int(\gamma)$.
\label{item:2}
\ee
\enlemma
\Proof
(i)\ For $x\in \ol A$ and $v\in \Int(\gamma)$, let us show
$x+v\in A+\Int(\gamma)$.
Since $x+v-\Int(\gamma)$ is a neighborhood of $x$,
there exists $y\in A\cap\bl x+v-\Int(\gamma)\br$.
Then, one has $x+v\in y+\Int(\gamma)\subset A+\Int(\gamma)$.

\spa
(ii)\ Let us show $U\subset U+\Int(\gamma)$.
For $x\in U$, there exists $y\in \Int(\gamma)\cap (x-U)$.
Then $x\in y+U\subset U+\Int(\gamma)$.

To complete it is enough to remark that $U+\gamma
\subset \bl U+\Int(\gamma)\br+\gamma=U+\Int(\gamma)$,
\QED
\begin{lemma}\label{le:gammaflat1}\hfill
\banum
\item The intersection of a $\gamma$-invariant set 
and a $\gamma^a$-invariant set is $\gamma$-flat,
i.e.,  $(B+\gamma)\cap(C+\gamma^a)$ is $\gamma$-flat
for any $B,C\subset\BBV$. 
\label{item:a}
\item
 If $A$ is $\gamma$-flat, then 
$\Int(A)=\bl A+\Int(\gamma)\br\cap\bl A+\Int(\gamma^a)\br$
and $\Int(A)$ is $\gamma$-flat.
\item  If $U$ is $\gamma$-open, then one has
$\Int(\ol U)=U$. 
\item if $U$ is a $\gamma$-flat open subset of $\BBV$, then
$Z\seteq (U+\gamma)\cap \ol{U+\gamma^a}$ is $\gamma$-locally closed and
$\Int(Z)=U$.
\eanum
\end{lemma}

\begin{proof}
(a) Set $D=(B+\gamma)\cap(C+\gamma^a)$. Then we have
$D+\gamma\subset (B+\gamma)+\gamma=B+\gamma$.
Similarly we have
$D+\gamma^a\subset (C+\gamma^a)+\gamma^a=C+\gamma^a$.
Hence, we obtain
$(D+\gamma)\cap (D+\gamma^a)\subset
(B+\gamma)\cap (C+\gamma^a)=D$.

\spa
(b)  Set $U=\bl A+\Int(\gamma)\br\cap\bl A+\Int(\gamma^a)\br$. Then $U$ is open and contained in $A$, hence, in $\Int(A)$. 
The other inclusion follows from Lemma~\ref{lem:Ugamma}\ \eqref{item:2}.

Note that $\Int(A)$ is $\gamma$-flat by (a).

\spa
(c) Let $x\in \Int(\ol U)$.
Then there exists $v\in\Int(\gamma^a)$ such that
$x+v\in \ol U$.
Since $x+\Int(\gamma^a)$ is a neighborhood of $x+v$, there exists
$y\in U\cap (x+\Int(\gamma^a))$.
Hence we have
$x\in y+\Int(\gamma)\subset U$. 

\spa
(d) It is obvious that $Z$ is $\gamma$-locally closed and $U\subset \Int(Z)$.
Conversely, (c) implies that
$\Int(Z)\subset (U+\gamma)\cap\Int(\ol{U+\gamma^a})
=(U+\gamma)\cap(U+\gamma^a)=U$.
\end{proof}

\begin{lemma}\label{le:gammaflat2}
Let  $Z$ be a  $\gamma$-locally closed subset of $\BBV$ and $\Omega=\Int(Z)$. 
Then 
\banum
\item for any $x\in Z$, there exists a neighborhood $W$ of $x$
such that $(x+\gamma^a)\cap W\subset Z$,
\item
if $x\in\BBV$ satisfies
$x\in\ol{\;(x+\gamma^a)\cap Z\;}$,
then one has $x\in Z$, \label{item:f}
\item
$Z$ is $\gamma$-flat,
\item $\Omega+\Int(\gamma)=Z+\Int(\gamma)=Z+\gamma$,\label{item:c}
\item
$Z=(\Omega+\gamma)\cap\ol{\;\Omega+\gamma^a\;}
=(\Omega+\gamma)\cap\ol{\Omega}$,
\item
$\Omega=(Z+\gamma)\cap\bl Z+\Int(\gamma^a)\br=Z\cap \bl Z+\Int(\gamma^a)\br$\\
$=\set{x\in \BBV}
{\hs{1ex}\parbox{42ex}{there exists an open neighborhood $W$ of $x$\\
such that $\bl x+\gamma\br\cap W\subset Z$}},$\label{item:omega}
\item
$\Omega$ is $\gamma$-flat.
\eanum
\end{lemma}

\begin{proof}
Write $Z=A\cap B$ with a $\gamma$-open $A$ and $\gamma$-closed $B$.
Hence $A+\gamma=A$ and $B+\gamma^a=B$.

\spa
(a)\ 
Setting $W=A$, one has $(x+\gamma^a)\cap W\subset B\cap A=Z$.

\spa 
(b) Assume that $x\in\ol{\;(x+\gamma^a)\cap Z\;}$.
Then one has $x\in\ol Z\subset B$.
Hence it remains to show $x\in A$.
Since $(x+\gamma^a)\cap Z\not=\emptyset$, there exists 
$y\in (x+\gamma^a)\cap Z$. Then one has
$x\in y+\gamma\subset A+\gamma=A$.

\spa
(c) follows from Lemma~\ref{le:gammaflat1}~\eqref{item:a}.

\spa
(d) It is enough to show that $Z\subset \Omega+\Int(\gamma)$.
Let $x\in Z$. Then $A\cap \bl x+\Int(\gamma^a)\br \subset A\cap B=Z$,
which implies that 
\eq
&&\text{$A\cap \bl x+\Int(\gamma^a)\br \subset \Omega$ for any $x\in Z$.}
\label{eq:Aomega}
\eneq
Since $A\cap \bl x+\Int(\gamma)^a\br\not=\emptyset$, there exists
$y\in A\cap \bl x+\Int(\gamma)^a\br$. Then one has
$x\in y+\Int(\gamma)\subset \Omega+\Int(\gamma)$.

\spa
(e) One has
$(\Omega+\gamma)\cap \ol {\,\Omega+\gamma^a\,} \subset A\cap B=Z$.
Let us show $Z\subset (\Omega+\gamma)\cap \ol\Omega$.
The inclusion $Z\subset \Omega+\gamma$ follows from \eqref{item:c}.
Finally, let us show that 
$x\in\ol \Omega$ for any $x\in Z$.
By  \eqref{eq:Aomega}, one has
$x\in\ol{\;A\cap\bl x+\Int(\gamma^a)\br\;}\subset\ol{\Omega}$.

\spa
(f) We shall show $(Z+\gamma)\cap\bl Z+\Int(\gamma^a)\br \subset \Omega\subset 
Z\cap \bl Z+\Int(\gamma^a)\br$.
We have
$$(Z+\gamma)\cap\bl Z+\Int(\gamma^a)\br
\subset(\Omega+\gamma)\cap\bl Z+\Int(\gamma^a)\br=A\cap B=Z.$$
Here the first inclusion follows from \eqref{item:c}. 
Hence we obtain $(Z+\gamma)\cap\bl Z+\Int(\gamma^a)\br \subset \Omega$.

The inclusion $\Omega\subset 
Z\cap \bl Z+\Int(\gamma^a)\br$ 
follows from $\Omega\subset\Omega+\Int(\gamma^a)$,
which is a consequence of Lemma~\ref{lem:Ugamma}~\eqref{item:2}. 

Let us prove the last equality.
Let $x\in\BBV$ such that 
$\bl x+\gamma\br\cap W\subset Z$
for an open neighborhood $W$ of $x$.
Take $y\in
\bl x+\Int(\gamma)\br\cap W\subset Z$.
Then we have
$x\in y+\Int(\gamma^a)\subset Z+\Int(\gamma^a)$. 

\spa
(g)\ 
One has
$$(\Omega+\gamma)\cap(\Omega+\gamma^a)
=(\Omega+\gamma)\cap\bl\Omega+\Int(\gamma^a)\br
\subset(\Omega+\gamma)\cap\bl Z+\Int(\gamma^a)\br
\subset\Omega.$$
Here the first equality is by Lemma~\ref{lem:Ugamma}~\eqref{item:2}
and the last inclusion follows from \eqref{item:omega}.
\end{proof}

The following proposition now follows from Lemma~\ref{le:gammaflat1}
and Lemma~\ref{le:gammaflat2}. 
\begin{proposition}
The set of $\gamma$-flat open subsets $\Omega$ of $\BBV$ and
the set of $\gamma$-locally closed subsets $Z$ of $\BBV$ are isomorphic 
by the correspondence 
$$\xymatrix@R=0ex@C=5ex{
\mbox{$\rule{4ex}{0ex}\Omega$}\ar@{|->}[r]&(\Omega+\gamma)\cap 
\ol{\;\Omega+\gamma^a\;}\\
\Int(Z)&\mbox{$Z.\rule{14ex}{0ex}$}\ar@{|->}[l]}
$$ 
\end{proposition}
\Lemma\label{lem:disjoint}
Let $U_i$ $(i=1,2)$ be $\gamma$-flat open subsets such that
$U_1\cap U_2=\emptyset$.
Set $Z_i=(U_i+\gamma)\cap\ol{\;U_i+\gamma^a\;}$.
Then one has $Z_1\cap Z_2=\emptyset$.
\enlemma
\Proof
Since $Z_i$ is $\gamma$-locally closed,
$Z_1\cap Z_2$ is also $\gamma$-locally closed. Then, by Lemma~\ref{le:gammaflat2}, 
$Z_1\cap Z_2$ is contained in the closure of 
$\Int(Z_1\cap Z_2)\subset\Int(Z_1)\cap \Int(Z_2)=U_1\cap U_2=\varnothing$. 
\QED

\subsection{Study of $\gamma$-sheaves}
In this subsection, we shall obtain some results on the behaviour of constructible $\gamma$-sheaves, preliminary to the study of piecewise linear  $\gamma$-sheaves and Theorem~\ref{th:linearcase}.

\Lemma\label{lem:inv}
Let $M$ be a manifold and $N$ a submanifold of $M$ of codimension $\ge2$.
Let $\Lambda$ be a closed conic involutive subset of $T^*M$.
If $\pi(\Lambda\cap\dTM)=N$, 
then $T^*_NM\subset\Lambda$.
\enlemma
\Proof
We choose a local coordinate system $(x)=(t,y)$ such that $N=\{t=0\}$, $(t)=(t_1,\ldots,t_m)$. Let $(x;\xi)=(t,y;\tau,\eta)$ denote the associated coordinates on $T^*M$. Let $(0,y_0)\in N$. There exists $(\tau_0,\eta_0)\neq0$ such that $(0,y_0;\tau_0,\eta_0)\in\Lambda$. Since $\Lambda\cap\dTM$ is contained in $t=0$, $\Lambda\cap\dTM$ is invariant by $\frac{\partial}{\partial \tau_k}$ for $1\le k\le m$  by \cite{KS90}*{Proposition~6.5.2}.  
Since $m>1$, this implies that $\Lambda\cap\dTM$ contains $(0,y_0;\tau,\eta_0)$ for any $\tau$. Moreover, since $\Lambda$ is conic, this implies $(0,y_0;\tau,0)\in \Lambda\cap\dTM$ for any $\tau$. 
\QED

\Prop\label{prop:germofF1b}
Let $F\in \mdrcg[\cor_{\BBV}]$. Then $S\seteq\Sing(F)$ has pure codimension $1$. 
Moreover, for any $x\in S_\reg$, one has $T_xS\cap \Int(\gamma)=\varnothing$,
or equivalently $(T_S^*\BBV)_x\subset\gamma^\circ\cup\gamma^{\circ a}$. 
\enprop
\begin{proof} 
Assume that there is a point where $S$ has codimension $\ge2$.
Take an open subset $U$ such that
$S\cap U$ is a non-empty submanifold of codimension $\ge 2$.
Note that $\Lambda\seteq\musupp(F)$ is involutive 
(\cite{KS90}*{Theorem~.6.5.4}), and
$U\cap\pi(\Lambda\cap\dTM)=S\cap U$. Hence Lemma~\ref{lem:inv} 
implies that
$\pi^{-1}U\cap T^*_SM\subset \Lambda\subset\BBV\times\gamma^{\circ a}$.
It contradicts the fact that $\gamma^{\circ a}$ is a proper closed convex cone.

The last assertion is a consequence of the fact that
$\musupp(F)\cap\pi^{-1}U\subset T^*_{S_\reg}\BBV$ for an open dense subanalytic subset $U$ of $S_\reg$
 and hence, $S$ being an hypersurface,  $(T^*_{S_\reg}\BBV\cap\pi^{-1}U)\subset
\musupp(F)\cup\musupp(F)^a$. 
\QED

From now on, and until the end of this paper, we assume that $\gamma$ satisfies~\eqref{hyp1b}.

\begin{lemma}\label{le:curveslpolyh}
Assume~\eqref{hyp1b}. Let $x\in\BBV$, let $I$ be an open interval of $\R$ with $0\in I$ and 
let $c\cl I\to\BBV$ be a real analytic  map. Assume that $c(t)\in (x+\gamma)\setminus\{x\}$ for $t\in I, t>0$ and $c(0)=x$. Then $c'(t)\in\gamma$ for all $t\geq0$ in a neighborhood of $0$.
\end{lemma}
\begin{proof}
Since $\gamma$ is  polyhedral, we may assume that
$\gamma=\set{x\in\BBV}{f(x)\ge0}$
for a linear function $f(x)$  on  $\BBV$.
Set $\phi(t)=f(c(t))$.
it is enough to prove that $\phi'(t)\geq0$
for $t\geq0$ in a neighborhood of $0$. If $\phi=0$, the result is clear. Otherwise, there exists $m\in\N$, $m>0$ such that  $\phi(t)=t^m v+{O}(t^{m+1})$ with $v\neq 0$.  Then $v>0$ and it follows that $\phi'(t)>0$ for $t>0$ in a neighborhood of $t=0$. 
\end{proof}

\begin{remark}
Lemma~\ref{le:curveslpolyh} is no more true without the assumption that the cone is polyhedral. 
Consider the cone 
$\gamma=\{(x,y,z)\in\R^3; x^2+y^2\leq z^2, z\geq0\}$
and the curve 
%$c(t)=(t\sqrt{1-t^2},t^2,t)$. One easily checks that $c'(t)\notin\gamma$ for $0<t<1$. 
$c(t)=(t\cos(t),t\sin(t), t)$.  One easily checks that $c'(t)\notin\gamma$ for $t>0$.
\end{remark}

\begin{theorem}\label{th:germofF}
Assume~\eqref{hyp1b}. 
Let $F\in \Derb_{\rcg}(\cor_{\BBV})$. Then for each $x\in\BBV$, there exists an open neighborhood $U$ of $x$ such that $F\vert_{(x+\gamma^a)\cap U}$ is constant. 
\end{theorem}
\begin{proof}
(i) By Corollary~\ref{cor:eqvderbg}, we may assume that $F$ is concentrated in degree $0$.
Moreover, the sheaf $F\tens\cor_{x+\gamma^a}$ belongs to $\mdrcg[\cor_\BBV]$. Hence we may assume from the beginning that $\supp(F)\subset x+\gamma^a$.  

\spa
(ii) By  Theorem~\ref{th:eqvderbg}, for any $y\in\BBV$, $F(y+\gamma)\isoto F_y$. 
 Since $F$ is $\R$-constructible, there exists an open neighborhood $V$ of $x$ such that $\sect(V;F)\isoto F_x$. Let $y\in V\cap(x+\gamma^a)$. Then $x\in y+\gamma$. Consider the diagram
\eqn
&&\xymatrix{
F(V)\,\,\,\ar@{>->}[r]\ar[d]_-\sim&F_y\isofrom F(y+\gamma)\ar[d]\\
F_x&F((y+\gamma)\cap V)\ar@{->>}[l]
}\eneqn
Let $E\eqdot F_x$, a finite-dimensional $\cor$-vector space. We have an injective map 
$E\to F_y$ for all $y\in (x+\gamma^a)\cap V$, hence a monomorphism $
E_{(x+\gamma^a)\cap V}\into F$. Define $G$ as the cokernel of this map. Then $G_x\simeq0$,
 $\supp(G)\subset x+\gamma^a$  and $G\in \mdrcg[\cor_\BBV]$. It remains to show that $G\simeq0$ in a neighborhood of $x$. 

\spa
(iii) Assume that  $G\neq0$ in any neighborhood of $x$. 
Then $\set{y\in \BBV}{G_y\not=0}$ 
is a subanalytic set whose closure contains $x$. 
By the curve selection lemma, 
 we find an analytic curve $c\cl I\to\gamma^a$ such that $c(0)=x$ and 
$G_{c(t)}\not=0$ for any $t\in I$ such that
$t>0$. 
By Lemma~\ref{le:curveslpolyh}, $c'(t)\in\gamma^a$ for all $t\geq0$ in a neighborhood of $0$. Setting $\phi(t)=c(t^2)$ for $t\geq0$ and $\phi(t)=x$ for $t\leq0$, we find a curve of class $C^1$ and $\supp(\opb{\phi}G)\subset\{t\geq0\}$. Denote by $(t;\tau)$ the homogeneous symplectic coordinates on $T^*\R$. 
 Applying~\cite{KS90}*{Cor.~6.4.4}, we get
 \eqn
 &&\left\{ \parbox{75ex}{
 $(0;\tau)\in\musupp(\opb{\phi}G)$ implies that there exists a sequence $\{(x_n;\xi_n)\}_n\subset \SSi(G)$ with 
$x_n\To[n]x$, $t_n\To[n]0$, $\langle\phi'(t_n),\xi_n\rangle\To[n]\tau$. 
 }\right.\eneqn
 Since $\phi'(t_n)\in\gamma^a$ and  $\xi_n\in\gammac$, 
we get $\tau\geq0$. Hence, $\opb{\phi}G\simeq0$ in a neighborhood of $0$. This is a contradiction. 
\end{proof}

For any sheaf $F\in\Derb_\Rc(\cor_\BBV)$, there exists a largest open subset $U$ of $\BBV$ such that  $F\vert_U$ is locally constant, namely the union of all open subsets on which $F$ is locally constant.
Moreover, $U$ is subanalytic in $\BBV$ since $U=\BBV\setminus\Sing(F)$. 
Note that $U\cap\supp(F)$ is again open in $\BBV$ and subanalytic. It is 
the  largest open subset of $\BBV$ on which  $F$ is locally constant 
with strictly positive rank. Hence it is a union of connected 
components of $U$.

\begin{corollary}\label{cor:germofF1}
Let $F\in \mdrcg[\cor_{\BBV}]$ and
let $U=\supp(F)\setminus\Sing(F)$.
Then $U$ is an open subset of\/
$\BBV$ and $F\vert_U$ is locally constant. Moreover,
$U$ is dense  in $\supp(F)$. 
In particular, one has $\supp(F)=\ol{\,\Int(\supp(F))\,}$.
\end{corollary}
\begin{proof}
We know already that $U$ is an open subset of $\BBV$ and
$F\vert_U$ is locally constant.
It remains to prove that $U$ is dense in $\set{x\in\BBV}{F_x\not=0}$.
We may assume that $\gamma$ is polyhedral. 
Let $x\in X$ such that $F_x\not\simeq0$. 
  Applying Theorem~\ref{th:germofF}, we find  an open neighborhood $W$ of $x$ such that 
$F$ is constant on $(x+\gamma^a)\cap W$, 
hence on the open set $V\eqdot(x+\Int(\gamma^a))\cap W$. 
Then  $x$ belongs to the closure of $V$ and $V\subset U$.
\end{proof}
\begin{corollary}\label{cor:germofF2}
Let $F\in  \mdrcg[\cor_{\BBV}]$ and $G\in\mdrcga[\cor_{\BBV}]$. Then $\ext{j}(G,F)\simeq0$ for $j\neq0$.
\end{corollary}
\begin{proof}
Let $U$ and $V$ be the largest open subsets of $\BBV$ on which   $F$ and $G$ are respectively  locally constant. Then $W=U\cap V$ is open, dense and $\rhom(G,F)$ is concentrated in degree $0$ on $W$. 
Therefore, $\supp(\ext{j}(G,F))$ has empty interior for $j>0$. Since this sheaf belongs to 
$\mdrcg[\cor_{\BBV}]$ by~\cite{KS90}*{Prop.~5.4.14} and Corollary~\ref{cor:eqvderbg}, it must be $0$  by 
Corollary~\ref{cor:germofF1}. 
 \end{proof}

 The result of Corollary~\ref{cor:germofF2} does not hold if both $F$ and $G$ belong to $\mdrcg[\cor_{\BBV}]$. 
 \begin{remark}\label{le:extn}
 One has $\rhom(\cor_{\gamma^a},\cor_{\Int(\gamma)}\,[n])\simeq\cor_{\{0\}}$.
  Indeed, Applying~\cite{KS90}*{Prop.~3.4.6}) we get 
\eqn
\rhom(\cor_{\gamma^a},\cor_{\Int(\gamma)}\,[n])&\simeq& 
\rhom(\cor_{\gamma^a},\RD'\cor_{\gamma})[n]\\
&\simeq& 
\rhom(\cor_{\gamma^a}\tens\cor_\gamma, \cor_\BBV)[n]\\
&\simeq& 
\rhom(\cor_{\{0\}}, \cor_\BBV)[n]
\simeq\cor_{\{0\}}.
\eneqn
 \end{remark}

\begin{theorem}\label{th:locctonflat}
Assume~\eqref{hyp1b}. 
 Let $\Omega$ be a $\gamma$-flat open set and let 
$Z=(\Omega+\gamma)\cap\ol{\;\Omega+\gamma^a\;}$,  a $\gamma$-locally closed subset. Let $F\in \Derb_{\rcg}(\cor_{\BBV})$ and 
assume that  $F\vert_\Omega$ is locally constant. Then $F\vert_Z$ is locally constant. 
\end{theorem}
\begin{proof} 
Let $x\in Z$ and let $U$ be an open convex neighborhood of $x$ such that $F(U)\isoto F_x$ and such that, applying Theorem~\ref{th:germofF}, $F$ is constant on $(x+\gamma^a)\cap U$. We choose a vector $v\in\Int(\gamma^a)$
and $\epsilon>0$ such that $[x,x+\epsilon v]\subset (x+\gamma^a)\cap U\subset Z$. 
Let $y=x+\epsilon v$. Then $y\in \Omega$ and $F$ is constant on $[x,y]$. 
Set $W\eqdot U\cap(y+\Int(\gamma))$ so that $x\in W$. 
We shall show that $F$ is constant on $W\cap Z$. 

Let $z\in W\cap Z$.
Then $[z,y]\subset \Omega+\gamma$, $y-z\in\gamma$ and
$]z,y]\subset\ol{(\Omega+\gamma^a)}+\Int(\gamma^a)=\Omega+ \gamma^a$.
 The last equality follows from Lemma~\ref{lem:Ugamma}. 
Hence we have
$]z,y]\subset \Omega$ and 
$F\vert_{[z,y]}$ is constant by Theorem~\ref{th:germofF}.
Consider the commutative diagram:
\eqn
&&\xymatrix{
F(U)\ar[d]_-\sim\ar[dr]\ar[drr]\ar[drrr]\ar[drrrr]\\
F_x&F({[x,y]})\ar[r]_-\sim\ar[l]^-\sim&F_y
&F({[z,y]})\ar[l]^-\sim\ar[r]_-\sim
&F_z\,.
}\eneqn 
The horizontal arrows are isomorphism since
$F\vert_{[x,y]}$ and $F\vert_{[z,y]}$ are constant. 
It follows that the map $F(U)\to F_z $ is an isomorphism.
\end{proof}

\subsection{Piecewise linear  $\gamma$-sheaves}\label{subsection:PLgamma}

\begin{definition}\label{def:strat}
A {\em \PL-$\gamma$-stratification}  $(A,Z)$ of a closed set $S$ is a $\gamma$-barcode $(A,Z)$ such that 
$\supp(A,Z)=S$ and $Z_\alpha\cap Z_\beta=\emptyset$ for any $\alpha,\beta\in A$ with $\alpha\neq \beta$.
\end{definition}

\Lemma\label{lem:gammastr}
Let $F\in \Derb_{\rcg}(\cor_\BBV)$ and let $(A,Z)$ be a \PL-$\gamma$-stratification of $\supp(F)$.
Then $F_x\simeq0$ for any $x\notin\bigcup_{\al\in A}Z_\al$.
\enlemma
\Proof
Assuming that $F_x\not\simeq 0$, let us show that
$x\in \bigcup_{\al\in A}Z_\al$.
By Theorem~\ref{th:germofF}, 
there exists an open neighborhood $U$ of $x$ such that 
$F\vert_{(x+\gamma^a)\cap U}$ is constant. 
Then one has
$(x+\gamma^a)\cap U\subset \supp(F)\subset\bigcup_{\al\in A}\ol{Z_\al}$.
Hence we obtain
\eqn
x\in \ol{\;\bl x+\Int(\gamma^a)\br\cap U\;}
\subset\ol{\;\bl x+\Int(\gamma^a)\br\cap\bl\bigcup_{\al\in A}\ol{Z_\al}\br\;}
&=&\bigcup_{\al\in A}\ol{\;\bl x+\Int(\gamma^a)\br\cap\ol{Z_\al}\;}\\
&=&\bigcup_{\al\in A}\ol{\;\bl x+\Int(\gamma^a)\br\cap Z_\al\;}.
\eneqn
Hence there exists $\al\in A$ such that
$x\in \ol{\;\bl x+\Int(\gamma^a)\br\cap Z_\al\;}$.
Then Lemma~\ref{le:gammaflat2}~\eqref{item:f} implies that
$x\in Z_\al$.
\QED

Recall that for $F\in\Derb_{\rcg}(\cor_\BBV)$, $\Sing(F)$ denotes its singular locus.

\begin{theorem}\label{th:linearcase}
Assume~\eqref{hyp1b}. Let $F\in\Derb_{\rclg}(\cor_\BBV)$. 
Then there exists a \PL-$\gamma$-stratification $(A,Z)$ of $\supp(F)$
such that $F\vert_{Z_\al}$ is constant for each $\al\in A$. 
\end{theorem}
\begin{proof}
(i) Assume first that $\supp(F)$ is compact.

\spa
(a) Since $F$ is \PL, there exists a finite family of
affine hyperplanes $\{H_a\}_{a\in A}$ such that
$\Sing(F)\subset\bigcup_{a\in A}H_a$. 
On the other hand, $\Sing(F)$ has pure codimension $1$, 
thanks to Proposition~\ref{prop:germofF1b}.  
Let $B=\{a\in A;\Int_{H_a}\bl\Sing(F)\cap H_a\br\neq\varnothing\}$. Then,
$\Sing(F)\cap\bl\bigcup_{a\in B}H_a\br$
is dense in $\Sing(F)$.
Hence we obtain
$\Sing(F)\subset\bigcup_{a\in B}H_a$.
Set $\Omega=\supp(F)\setminus\bl\bigcup_{a\in B}H_a\br$.
Then $\Omega$ is an open subset of $\BBV$ and dense in $\supp(F)$ 
by Corollary~\ref{cor:germofF1}.

\spa
(b) Let $\Omega=\bigsqcup_{i\in I}\Omega_i$ be the decomposition of $\Omega$ into  connected components. 
Then each $\Omega_i$ is  an open convex polytope. 

At generic points of $H_a$ ($a\in B$), one has $\dT{}^*_{H_a}\BBV\cap\musupp(F)\neq\varnothing$  by Proposition~\ref{prop:germofF1b}.   If $H_a=\{\langle x,\xi\rangle=c\}$, one has $\pm\xi\in\gammac$ and thus  $H_a\cap(x+\Int(\gamma))=\varnothing$ for any $x\in H_a$.
Denote by $H_a^\pm$ the two open half-spaces with boundary $H_a$. These are $\gamma$-flat open sets and it follows that any connected component $\Omega_i$ of 
$\Omega$, which is a finite intersection of such half-spaces,  
is also $\gamma$-flat.

Set $Z_i=(\Omega_i+\gamma)\cap\ol{\Omega_i+\gamma^a}$. Then
each $Z_i$ is $\gamma$-locally closed and
$\supp(F)=\bigcup_{i\in I}\ol{Z_i}$.
By Lemma~\ref{lem:disjoint}, $Z_i\cap Z_j=\emptyset$ if $i\not=j$.
Hence $\{Z_i\}_{i\in I}$ is a $\gamma$-stratification of $\supp(F)$.
By Proposition~\ref{th:locctonflat},
$F\vert_{Z_i}$ is locally constant.  Since $Z_i$ is convex, 
$F\vert_{Z_i}$ is constant.

\spa
(ii) Now we consider the general case
where  $\supp(F)$ is not necessarily compact.

Taking $v\in\Int(\gamma)$,
one sets $U_n=-nv+\Int(\gamma)$ and $S_n=nv+\gamma^a$ for $n\in \Z$.
Then $\{U_n\}_{n\in\Z}$ is an increasing family of
 $\gamma$-open subsets, and
$\{S_n\}_{n\in\Z}$ is an increasing family of
 $\gamma$-closed subsets.
Moreover, one has $\BBV=\bigcup_{n\in\Z}U_n=\bigcup_{n\in\Z}S_n$
and $\bigcap_{n\in\Z}U_n=\bigcap_{n\in\Z}S_n=\emptyset$. 
Set $I=\Z\times\Z$ and, for $i=(m,n)\in I$, set
$K_i\seteq(U_m\setminus U_{m-1})\cap (S_n\setminus S_{n-1})$. 
Then $\{K_i\}_{i\in I}$ is a locally finite family of
$\gamma$-locally closed subsets
such that $\BBV=\bigsqcup_{i\in I}K_i$.

Set $F_i=F\tens\cor_{K_i}$.
Then $F_i\in\Derb_{\rclg}(\cor_\BBV)$
and $\supp(F_i)$ is compact.
Hence by Step (i) there exists a finite \PL-$\gamma$-stratification
$\{Z_\al\}_{\al\in A_i}$ of $\supp(F_i)$ such that
$F_i\vert_{Z_\al}$ is constant for each $\al\in A_i$.
Then, setting $A=\bigsqcup A_i$, we obtain a desired
\PL-$\gamma$-stratification $\{Z_\al\}_{\al\in A}$ of $\supp(F)$.
Note that $Z_\al\cap Z_{\al'}=\emptyset$ for $\al\not=\al'$
follows again from Lemma~\ref{lem:disjoint}.
Indeed, assume that $\al\in A_i$ and $\al'\in A_{i'}$.
If $i=i'$, one has obviously $Z_\al\cap Z_{\al'}=\emptyset$.
If $i\not=i'$, then
$\Int(Z_\al)\cap\Int(Z_{\al'})
\subset K_i\cap K_{i'}=\emptyset$, and 
Lemma~\ref{lem:disjoint} implies that $Z_\al\cap Z_{\al'}=\emptyset$.
\end{proof}

\begin{remark}
In the course of the proof of Theorem~\ref{th:linearcase}, we have also obtained the following result.

Let $F\in\Derb_{\rcg}(\cor_\BBV)$, and let  $\{H_a\}_{a\in A}$ be
a locally finite family of affine hyperplanes
such that $\Sing(F)\subset\bigcup_{a\in A}H_a$.
Then, one has
$\Sing(F)\subset\bigcup_{a\in B}H_a$
where $B=\set{a\in A}{\Int_{H_a}\bl\Sing(F)\cap H_a\br\neq\varnothing}$. 
Moreover, $F$ is $\PL$ by Lemma~\ref{lem:gammastr}. 
\end{remark}
As usual, for an additive category $\shc$,  we shall denote by $\RC^\rb(\shc)$ the category of bounded complexes of objects of $\shc$.
\begin{conjecture}
Let $F\in\Derb_{\rclg}(\cor_\BBV)$ and assume that $F$ has compact support. Then there exists a bounded complex 
$F^\scbul\in \RC^\rb(\mdrclg[\cor_\Vg])$ whose image in  $\Derb_{\rclg}(\cor_\BBV)$  is isomorphic to $F$ and such that each component $F^j$ 
of $F^\scbul$ is a barcode $\gamma$-sheaf (see Definition~\ref{def:barcodesheaf}) with compact support. 
\end{conjecture}

\vspace*{1cm}
\noindent
\parbox[t]{21em}
{\scriptsize{
\noindent
Masaki Kashiwara\\
 Kyoto University \\
Research Institute for Mathematical Sciences\\
Kyoto, 606--8502, Japan\\
and\\
Department of Mathematical Sciences and School of Mathematics,\\
Korean Institute for Advanced Studies, \\
Seoul 130-722, Korea

\medskip\noindent
Pierre Schapira\\
Sorbonne Universit{\'e}s, UPMC Univ Paris 6\\
Institut de Math{\'e}matiques de Jussieu\\
e-mail: pierre.schapira@imj-prg.fr\\
http://webusers.imj-prg.fr/\textasciitilde pierre.schapira/
}}

\end{document}